\documentclass[12pt,leqno]{amsart}
\usepackage{latexsym}
\usepackage{amsmath}
\usepackage{amsthm}
\usepackage{amssymb}
\usepackage{amsfonts}

\usepackage{comment}
\begin{document}
\theoremstyle{plain}
\newtheorem{MainThm}{Theorem}
\newtheorem{thm}{Theorem}[section]
\newtheorem{clry}[thm]{Corollary}
\newtheorem{prop}[thm]{Proposition}
\newtheorem{lemma}[thm]{Lemma}
\newtheorem{deft}[thm]{Definition}
\newtheorem{hyp}{Assumption}
\newtheorem*{conjecture}{Conjecture}
\newtheorem{question}{Question}

\newtheorem{claim}[thm]{Claim}

\theoremstyle{definition}
\newtheorem*{definition}{Definition}
\newtheorem{assumption}{Assumption}
\newtheorem{rem}[thm]{Remark}
\newtheorem*{remark}{Remark}
\newtheorem*{acknow}{Acknowledgements}

\newtheorem{example}[thm]{Example}
\newtheorem*{examplenonum}{Example}
\numberwithin{equation}{section}
\newcommand{\nocontentsline}[3]{}
\newcommand{\tocless}[2]{\bgroup\let\addcontentsline=\nocontentsline#1{#2}\egroup}
\newcommand{\eps}{\varepsilon}
\renewcommand{\phi}{\varphi}
\renewcommand{\d}{\partial}
\newcommand{\re}{\mathop{\rm Re} }
\newcommand{\im}{\mathop{\rm Im}}
\newcommand{\mR}{\mathbb{R}}
\newcommand{\mC}{\mathbb{C}}
\newcommand{\mN}{\mathbb{N}} 
\newcommand{\mZ}{\mathbb{Z}} 
\newcommand{\mK}{\mathbb{K}}
\newcommand{\supp}{\mathop{\rm supp}}
\newcommand{\abs}[1]{\lvert #1 \rvert}
\newcommand{\norm}[1]{\lVert #1 \rVert}
\newcommand{\csubset}{\Subset}
\newcommand{\detg}{\lvert g \rvert}
\newcommand{\msetminus}{\setminus}

\newcommand{\br}[1]{\langle #1 \rangle}

\newcommand{\ehat}{\,\hat{\rule{0pt}{6pt}}\,}
\newcommand{\echeck}{\,\check{\rule{0pt}{6pt}}\,}
\newcommand{\etilde}{\,\tilde{\rule{0pt}{6pt}}\,}

\newcommand{\tr}{\mathrm{tr}}
\newcommand{\mdiv}{\mathrm{div}}

\title[Calder\'on problem with partial data]{Recent progress in the Calder\'on problem with partial data}

\author{Carlos Kenig}
\address{Department of Mathematics, University of Chicago}
\email{cek@math.uchicago.edu}

\author{Mikko Salo}
\address{Department of Mathematics and Statistics, University of Jyv\"askyl\"a}
\email{mikko.j.salo@jyu.fi}


\begin{abstract}
We survey recent results on Calder\'on's inverse problem with partial data, focusing on three and higher dimensions.
\end{abstract}

\maketitle

\tableofcontents

\section{Introduction} \label{sec_intro}

The Calder\'on problem with partial data asks to determine the electrical conductivity of a body from electrical measurements made on certain subsets of the boundary. This is a fundamental inverse problem, also mentioned as an open question in Gunther Uhlmann's ICM address \cite{U_ICM}. Subsequent years have seen several advances in partial data problems, many of them due to Gunther Uhlmann and his collaborators, and it is a pleasure for us to survey some of these developments in this volume in honor of Gunther's remarkable career.

Let us give the formal statement of the problem. If $\Omega \subset \mR^n$, $n \geq 2$, is a bounded domain with $C^{\infty}$ boundary, if $\gamma \in L^{\infty}(\Omega)$ is a positive function (the electrical conductivity of the medium), and if $\Gamma_D, \Gamma_N$ are open subsets of $\partial \Omega$, we consider the partial Cauchy data set 
\begin{align*}
C_{\gamma}^{\Gamma_D,\Gamma_N} = \{ (u|_{\Gamma_D}, \gamma \partial_{\nu} u|_{\Gamma_N}) &\,;\, \mdiv(\gamma \nabla u) = 0 \text{ in } \Omega, \ u \in H^1(\Omega), \\
 & \quad \supp(u|_{\partial \Omega}) \subset \Gamma_D \}.
\end{align*}
In the Calder\'on problem with partial data, the objective is to determine the conductivity $\gamma$ from the knowledge of $C_{\gamma}^{\Gamma_D, \Gamma_N}$ for given (possibly very small) sets $\Gamma_D, \Gamma_N$. Here $\partial_{\nu}$ is the normal derivative, and the conormal derivative $\gamma \partial_{\nu} u|_{\partial \Omega}$ is interpreted in the weak sense as an element of $H^{-1/2}(\partial \Omega)$. Thus we have 
$$
C_{\gamma}^{\Gamma_D,\Gamma_N} \subset H^{1/2}(\Gamma_D) \times H^{-1/2}(\Gamma_N).
$$

It is very useful to consider the related inverse boundary value problem for the Schr\"odinger equation with partial data. There, the objective is to determine a potential $q \in L^{\infty}(\Omega)$ from the partial Cauchy data set 
\begin{align*}
C_{q}^{\Gamma_D,\Gamma_N} = \{ (u|_{\Gamma_D}, \partial_{\nu} u|_{\Gamma_N}) &\,;\, (-\Delta+q) = 0 \text{ in } \Omega, \ u \in H_{\Delta}(\Omega), \\
 & \quad \supp(u|_{\partial \Omega}) \subset \Gamma_D \}.
\end{align*}
Here, the space $H_{\Delta}(\Omega)$ is defined by 
$$
H_{\Delta}(\Omega) = \{ u \in L^2(\Omega) \,;\, \Delta u \in L^2(\Omega) \}.
$$
It is known (see \cite{BU}) that for $u$ in this space, the trace $u|_{\partial \Omega}$ and normal derivative $\partial_{\nu} u|_{\partial \Omega}$ are in $H^{-1/2}(\partial \Omega)$ and $H^{-3/2}(\partial \Omega)$, respectively. Therefore, 
$$
C_{q}^{\Gamma_D,\Gamma_N} \subset H^{-1/2}(\Gamma_D) \times H^{-3/2}(\Gamma_N).
$$

We formulate the relevant partial data problems as follows:

\begin{question}
Let $\Gamma_D, \Gamma_N$ be open subsets of $\partial \Omega$ and let $\gamma_1, \gamma_2 \in L^{\infty}(\Omega)$. If 
$$
C_{\gamma_1}^{\Gamma_D,\Gamma_N} = C_{\gamma_2}^{\Gamma_D,\Gamma_N},
$$
is it true that $\gamma_1 = \gamma_2$?
\end{question}

\begin{question}
Let $\Gamma_D, \Gamma_N$ be open subsets of $\partial \Omega$ and let $q_1, q_2 \in L^{\infty}(\Omega)$. If 
$$
C_{q_1}^{\Gamma_D,\Gamma_N} = C_{q_2}^{\Gamma_D,\Gamma_N},
$$
is it true that $q_1 = q_2$?
\end{question}

In both problems above, we think of $u|_{\partial \Omega}$ as Dirichlet data (voltage) prescribed only on $\Gamma_D$, and we assume that one can measure the Neumann data of the corresponding solution (outgoing current) on $\Gamma_N$. The Cauchy data set is determined by the Dirichlet-to-Neumann map (DN map) $\Lambda_{\gamma}: H^{1/2}(\partial \Omega) \to H^{-1/2}(\partial \Omega)$, defined by  
$$
\Lambda_{\gamma}: u|_{\partial \Omega} \mapsto \gamma \partial_{\nu} u|_{\partial \Omega} \ \ \text{where $u \in H^1(\Omega)$ solves $\mdiv(\gamma \nabla u) = 0$ in $\Omega$}.
$$
The partial Cauchy data set is obtained from the graph of $\Lambda_{\gamma}$ as 
$$
C_{\gamma}^{\Gamma_D,\Gamma_N} = \{ (f|_{\Gamma_D}, \Lambda_{\gamma} f|_{\Gamma_N}) \,;\, f \in H^{1/2}(\partial \Omega), \ \supp(f) \subset \Gamma_D \}.
$$
Also $C_{q}^{\Gamma_D,\Gamma_N}$ is obtained by restricting the graph of the DN map $\Lambda_q: u|_{\partial \Omega} \mapsto \partial_{\nu} u|_{\partial \Omega}$, where $(-\Delta+q)u = 0$ in $\Omega$, provided that $0$ is not a Dirichlet eigenvalue of $-\Delta+q$ in $\Omega$.

One can think of three subcases of the above problems:
\begin{itemize}
\item
\emph{Full data}: $\Gamma_D = \Gamma_N = \partial \Omega$.
\item
\emph{Local data}: $\Gamma_D = \Gamma_N = \Gamma$, where $\Gamma$ can be any nonempty open subset of $\partial \Omega$.
\item
\emph{Data on disjoint sets}: $\Gamma_D$ and $\Gamma_N$ are disjoint open sets.
\end{itemize}

The most complete results are of course available for the full data case. Moreover, virtually all known results involve some version of the method of complex geometrical optics (CGO) solutions. This method has its origin in the works \cite{Faddeev}, \cite{C}, and major results for full data inverse boundary value problems based on the CGO method are \cite{SU}, \cite{HT} in dimensions $n \geq 3$ and \cite{N_2D}, \cite{AP}, \cite{Bu} in the case $n=2$. In particular, it has been proved that the set $C_{\gamma}^{\partial \Omega, \partial \Omega}$ determines uniquely a conductivity $\gamma \in C^1(\overline{\Omega})$ if $n \geq 3$ and a conductivity $\gamma \in L^{\infty}(\Omega)$ if $n = 2$. See the survey \cite{U_IP} for further references.

For the partial data question where the sets $\Gamma_D$ or $\Gamma_N$ may not be the whole boundary, we mention here four main approaches. Formulated in terms of the Schr\"odinger problem, it is known that $C_q^{\Gamma_D,\Gamma_N}$ determines $q$ in $\Omega$ in the following cases:
\begin{enumerate}
\item
$n \geq 3$, the set $\Gamma_D$ is possibly very small, and $\Gamma_N$ is slightly larger than $\partial \Omega \setminus \Gamma_D$ (Kenig, Sj\"ostrand, and Uhlmann \cite{KSU}) \\
\item
$n \geq 3$ and $\Gamma_D = \Gamma_N = \Gamma$, and $\partial \Omega \setminus \Gamma$ is either part of a hyperplane or part of a sphere (Isakov \cite{I}) \\
\item
$n = 2$ and $\Gamma_D = \Gamma_N = \Gamma$, where $\Gamma$ can be an arbitrary open subset of $\partial \Omega$ (Imanuvilov, Uhlmann, and Yamamoto \cite{IUY}) \\
\item
$n \geq 2$, linearized partial data problem, $\Gamma_D = \Gamma_N = \Gamma$ where $\Gamma$ can be an arbitrary open subset of $\partial \Omega$ (Dos Santos, Kenig, Sj\"ostrand, and Uhlmann \cite{DKSjU09})
\end{enumerate}

Here approach (1) requires roughly complementary sets $\Gamma_D$ and $\Gamma_N$, whereas approaches (2)--(4) deal with the local data problem. Approaches (1)--(3) also give a partial data result for the conductivity equation with the same assumptions on the dimension and the sets $\Gamma_D, \Gamma_N$. In (4), the linearized partial data problem asks to show injectivity of the Fr\'echet derivative of $\Lambda_q$ at $q=0$ instead of injectivity of the full map $q \mapsto \Lambda_q$, when restricted to the sets $\Gamma_D$ and $\Gamma_N$.

It is interesting that each of the four approaches is based on a version of CGO solutions, but still the approaches are distinct in the sense that none of the above results is contained in any of the others. Approach (1) is based on Carleman estimates with boundary terms for limiting Carleman weights, whereas approach (2) is based on reflection arguments and the full data methods of \cite{SU}. Approach (3) uses limiting Carleman weights with critical points and stationary phase, but involves complex analysis and is therefore restricted to two dimensions. Approach (4) is based on analytic microlocal analysis but so far only works for the linearized case.

In a recent work \cite{KSa}, we unified the approaches (1) and (2) and extended both of them, giving new partial data results in dimensions $n \geq 3$ also on certain Riemannian manifolds. Below, we will explain the approaches (1) and (2) from the new perspective obtained from \cite{KSa}, and we will also give detailed proofs of certain partial data results in \cite{KSa} restricting to the special case of Euclidean domains. We hope that the present treatment will be more accessible to readers familiar with Euclidean analysis than the geometric paper \cite{KSa}. Approaches (3) and (4) give rather final results in the local data problem for $n=2$ and for the linearized case. We refer to the recent survey \cite{GT_survey} on two-dimensional partial data problems for more details on approach (3). Approach (4) will be discussed in Section \ref{sec_linearized}, and some open problems are listed in Section \ref{sec_open}.

We list here some further references for partial data results, first for the case $n \geq 3$. The Carleman estimate approach was initiated in \cite{BU} and \cite{KSU}. Based on this approach, there are low regularity results \cite{Knudsen}, \cite{Zhang}, results for other scalar equations \cite{DKSjU07}, \cite{KnudsenSalo}, \cite{Chung}, \cite{Chung_ND} and systems \cite{SaloTzou}, stability results \cite{HeckWang}, \cite{HeckWang2}, \cite{Tzou}, \cite{CDR}, and reconstruction results \cite{NS}. The reflection approach was introduced in \cite{I}, and has been employed for the Maxwell system \cite{COS}, \cite{Caro_stability}. Partial data results for slab geometries are given in \cite{LiUhlmann}, \cite{KLU}. A result analogous to Theorem \ref{thm_ks1} was independently obtained in \cite{IY_3D}.

In two dimensions, the main partial data result is \cite{IUY} which has been extended to more general equations \cite{IUY_general}, combinations of measurements on disjoint sets \cite{IUY_disjoint}, less regular coefficients \cite{IY_nonsmooth}, and some systems \cite{IY_systems}. An earlier result is in \cite{ALP}. In the case of Riemann surfaces with boundary, corresponding partial data results are given in \cite{GT}, \cite{GT_magnetic}, \cite{GT_system}.

For piecewise analytic conductivities, uniqueness in the local data problem follows from boundary determination results \cite{KV}, \cite{KV2}. See \cite{Gebauer} for a related construction. An early result for the linearized problem in an annular domain in $\mR^2$ with no measurements on the inner boundary is in \cite{H98}. In the case when the conductivity is known near the boundary, the partial data problem can be reduced to the full data problem \cite{I88}, \cite{AU}, \cite{Fathallah}, \cite{Benjoud}, \cite{Alessandrini_partial}, \cite{HPS}. Also, we remark that in the corresponding problem for the wave equation, it has been known for a long time (see \cite{KKL}) that measuring the Dirichlet and Neumann data of waves on an arbitrary open subset of the boundary is sufficient to determine the coefficients uniquely up to natural gauge transforms. Recent partial results for the case where Dirichlet and Neumann data are measured on disjoint sets are in \cite{LO1}, \cite{LO2}.

The paper is structured as follows. Section \ref{sec_intro} is the introduction. In Section \ref{sec_results} we give precise statements for the various partial data results in the literature, concentrating on the Schr\"odinger problem. Section \ref{sec_strategy} describes the main strategy for proving most of these results, by reducing to a density statement for products of solutions that vanish on parts of the boundary. A Carleman estimate with boundary terms relevant for partial data results is proved in Section \ref{sec_carleman}. This is a special case of the corresponding estimate in \cite{KSa}, but we give self-contained proofs in the Euclidean case. Section \ref{sec_cgo} discusses the construction of CGO solutions vanishing on parts of the boundary, and Section \ref{sec_uniqueness} explains the corresponding uniqueness results relying on the injectivity of a mixed Fourier transform/local attenuated geodesic ray transform of the unknown coefficient. The linearized problem is considered in Section \ref{sec_linearized}, and Section \ref{sec_open} lists open questions related to the partial data problem.


\bigskip

\noindent {\bf Acknowledgements.} \ 
C.K.~is partly supported by NSF, and M.S.~is supported in part by the Academy of Finland and an ERC Starting Grant.

\section{Partial data results} \label{sec_results}

In this section we give precise statements of the partial data results mentioned in the introduction. We will only consider the Schr\"odinger problem, since the conductivity problem can be reduced to that case by the well known relation  
\begin{gather*}
\Lambda_{q_{\gamma}} f = \gamma^{-1/2} \Lambda_{\gamma}(\gamma^{-1/2} f) + \frac{1}{2} \gamma^{-1} (\partial_{\nu} \gamma) f, \qquad q_{\gamma} = \frac{\Delta(\gamma^{1/2})}{\gamma^{1/2}}.
\end{gather*}
This reduction works if the conductivities are sufficiently regular (one may additionally need to assume that the two conductivities agree on part of the boundary).

It will be useful to introduce some notation. In all results below, we assume that $\Omega \subset \mR^n$ is a bounded open connected set with $C^{\infty}$ boundary. If $\phi \in C^{\infty}(\overline{\Omega})$ is a real valued function, we can write the boundary of $\Omega$ as the disjoint union 
$$
\partial \Omega = \partial \Omega_+ \cup \partial \Omega_- \cup \partial \Omega_{\rm tan}
$$
where 
\begin{gather*}
\partial \Omega_{\pm} = \partial \Omega_{\pm}(\phi) = \{ x \in \partial \Omega \,;\, \pm \nabla \phi(x) \cdot \nu(x) > 0 \}, \\
\partial \Omega_{\rm tan} = \partial \Omega_{\rm tan}(\phi) = \{ x \in \partial \Omega \,;\, \nabla \phi(x) \cdot \nu(x) = 0 \}.
\end{gather*}
Here $\nu$ is the outer unit normal of $\partial \Omega$.

The functions $\phi$ that can be used in partial data results are typically \emph{limiting Carleman weights}. This is a special class of functions, introduced in \cite{KSU}, that coincides with the set of harmonic functions (with some restriction on their critical points) if $n=2$, and if $n \geq 3$ it consists of the following six functions up to translation and scaling:
\begin{gather*}
\alpha \cdot x, \ \ \log \,\abs{x}, \ \ \frac{\alpha \cdot x}{\abs{x}^2}, \ \ \text{arg}((\alpha + i\beta) \cdot x), \\
\text{arg}(e^{i\theta}(x+i\xi)^2), \ \ \log \frac{\abs{x+\xi}^2}{\abs{x-\xi}^2}.
\end{gather*}
Here $\alpha$ and $\beta$ are orthogonal unit vectors, $\theta \in [0,2\pi)$, $\xi \in \mR^n \setminus \{ 0 \}$, and the argument function is defined by 
$$
\text{arg}(z) = 2 \arctan \frac{\im(z)}{\abs{z}+\re(z)}, \quad z \in \mC \setminus (-\infty,0].
$$
We refer to \cite{DKSaU} for more information and a thorough analysis of limiting Carleman weights that have no critical points.

It is suggested by the methods of \cite{KSa} that for a fixed limiting Carleman weight $\phi$, measuring Neumann data on $\partial \Omega_+(\phi)$ for Dirichlet data input of $\partial \Omega_-(\phi)$, with no measurements required on $\partial \Omega_{\rm tan}(\phi)$, might be sufficient for determining the unknown coefficients (and this should also hold with $\partial \Omega_+(\phi)$ and $\partial \Omega_-(\phi)$ interchanged). All results described below can be understood in light of this idea, but most of them require measurements on $\partial \Omega_{\pm}(\phi)$ for a large family of different $\phi$'s instead of just one $\phi$.

\bigskip

\noindent 2.1. {\bf The result of \cite{KSU}.}
This result is stated in terms of the front and back faces of $\partial \Omega$ with respect to some point $x_0$ which is outside the convex hull ${\rm ch}(\overline{\Omega})$ of $\overline{\Omega}$. Note that if $\Omega$ is strictly convex, the front face can be made arbitrarily small by placing $x_0$ close to the boundary, but in this case the back face will be very large.

\begin{thm} \label{thm_ksu1}
Let $\Omega \subset \mR^n$, $n \geq 3$. If $x_0 \in \mR^n \setminus {\rm ch}(\overline{\Omega})$, define the front and back face of $\partial \Omega$ by 
\begin{gather*}
F(x_0) = \{ x \in \partial \Omega \,;\, (x-x_0) \cdot \nu(x) \leq 0 \}, \\
B(x_0) = \{ x \in \partial \Omega \,;\, (x-x_0) \cdot \nu(x) \geq 0 \}.
\end{gather*}
Let $\Gamma_D, \Gamma_N$ be open subsets of $\partial \Omega$ with $F(x_0) \subset \Gamma_D$ and $B(x_0) \subset \Gamma_N$. If $q_1, q_2 \in L^{\infty}(\Omega)$ and if 
$$
C_{q_1}^{\Gamma_D,\Gamma_N} = C_{q_2}^{\Gamma_D,\Gamma_N},
$$
then $q_1 = q_2$.
\end{thm}

We also state another partial data result, which follows from the previous theorem by placing $x_0$ sufficiently far from $\Omega$. This is close to the earlier result of \cite{BU}.

\begin{thm} \label{thm_ksu2}
Let $\Omega \subset \mR^n$, $n \geq 3$. If $\alpha \in \mR^n$ is a unit vector, define the front and back face of $\partial \Omega$ by 
\begin{gather*}
F(\alpha) = \{ x \in \partial \Omega \,;\, \alpha \cdot \nu(x) \leq 0 \}, \\
B(\alpha) = \{ x \in \partial \Omega \,;\, \alpha \cdot \nu(x) \geq 0 \}.
\end{gather*}
Let $\Gamma_D, \Gamma_N$ be open subsets of $\partial \Omega$ with $F(\alpha) \subset \Gamma_D$ and $B(\alpha) \subset \Gamma_N$. If $q_1, q_2 \in L^{\infty}(\Omega)$ and if 
$$
C_{q_1}^{\Gamma_D,\Gamma_N} = C_{q_2}^{\Gamma_D,\Gamma_N},
$$
then $q_1 = q_2$.
\end{thm}

Both results can be understood as follows: partial measurements are sufficient for determining the potential provided that $\Gamma_D$ and $\Gamma_N$ satisfy 
$$
\partial \Omega_-(\phi) \cup \partial \Omega_{\rm tan}(\phi) \subset \Gamma_D, \qquad \partial \Omega_+(\phi) \cup \partial \Omega_{\rm tan}(\phi) \subset \Gamma_N
$$
for a suitable family of limiting Carleman weights. In the first theorem $\phi(x) = \log\,\abs{x-x_0'}$ where $x_0'$ ranges over a small neighborhood of $x_0$, and in the second theorem $\phi(x) = \alpha' \cdot x$ where $\alpha'$ ranges over a small neighborhood of $\alpha$ in the unit sphere.

\bigskip

\noindent 2.2. {\bf The result of \cite{I}.} This is a local data result where one measures both Dirichlet and Neumann data on the same set $\Gamma = \Gamma_D = \Gamma_N$, but where the inaccessible part $\Gamma_0 = \partial \Omega \setminus \Gamma$ has to be part of a hyperplane.

\begin{thm} \label{thm_i1}
Let $\Omega \subset \mR^n$, $n \geq 3$. Assume that $\Omega \subset \{ x_n > 0 \}$, let $\Gamma$ be a nonempty open subset of $\partial \Omega$, and assume that $\Gamma_0 = \partial \Omega \setminus \Gamma$ satisfies $\Gamma_0 \subset \{ x_n = 0 \}$. If $q_1, q_2 \in L^{\infty}(\Omega)$ and if 
$$
C_{q_1}^{\Gamma,\Gamma} = C_{q_2}^{\Gamma,\Gamma},
$$
then $q_1 = q_2$.
\end{thm}

By applying the Kelvin transform $K(x) = x/\abs{x}^2$, this theorem implies a similar result (also proved in \cite{I}) where the inaccessible part of the boundary has to be part of a sphere.

\begin{thm} \label{thm_i2}
Let $\Omega \subset \mR^n$, $n \geq 3$. Assume that $\Omega$ is a strict subset of some ball $B$, let $\Gamma$ be a nonempty open subset of $\partial \Omega$, and assume that $\Gamma_0 = \partial \Omega \setminus \Gamma$ satisfies $\Gamma_0 \subset \partial B$. If $q_1, q_2 \in L^{\infty}(\Omega)$ and if 
$$
C_{q_1}^{\Gamma,\Gamma} = C_{q_2}^{\Gamma,\Gamma},
$$
then $q_1 = q_2$.
\end{thm}

These results can be understood so that measurements on $\Gamma$ are sufficient to determine the potential if the inaccessible part $\Gamma_0$ satisfies certain geometric conditions related to limiting Carleman weights. In the first theorem, it holds that 
$$
\Gamma_0 \subset \bigcap_{\underset{\alpha \cdot e_n = 0}{\alpha \in \mR^n, \abs{\alpha}=1}}\partial \Omega_{\rm tan}(\phi_{\alpha})
$$
where $\phi_{\alpha}(x) = \alpha \cdot x$ and $e_n$ is the $n$th coordinate vector. Taking complements, this condition means that $\Gamma$ should be sufficiently large in the sense that  
$$
\bigcup_{\underset{\alpha \cdot e_n = 0}{\alpha \in \mR^n, \abs{\alpha}=1}} (\partial \Omega_+(\phi_{\alpha})\cup \partial \Omega_-(\phi_{\alpha})) \subset \Gamma.
$$
In the second theorem, if the coordinates are normalized in such a way that $B = B(\frac{1}{2}e_n, \frac{1}{2})$, one has instead 
$$
\Gamma_0 \subset \bigcap_{\underset{\alpha \cdot e_n = 0}{\alpha \in \mR^n, \abs{\alpha}=1}}\partial \Omega_{\rm tan}(K^* \phi_{\alpha})
$$
where $K$ is the Kelvin transform and $K^* \varphi_{\alpha}$ is the limiting Carleman weight $K^* \varphi_{\alpha}(x) = \alpha \cdot x/\abs{x}^{2}$.

\bigskip

\noindent 2.3. {\bf The result of \cite{IUY}.} This is a general local data result that is valid for two-dimensional domains. (It was extended to potentials with $W^{1,p}$, $p > 2$, regularity in \cite{IY_nonsmooth}.)

\begin{thm} \label{thm_iuy}
Let $\Omega \subset \mR^2$ be a bounded open set with $C^{\infty}$ boundary, and let $\Gamma$ be a nonempty open subset of $\partial \Omega$. If $q_1, q_2 \in C^{4,\alpha}(\overline{\Omega})$ for some $\alpha > 0$ and if 
$$
C_{q_1}^{\Gamma,\Gamma} = C_{q_2}^{\Gamma,\Gamma},
$$
then $q_1 = q_2$.
\end{thm}

The result is related to limiting Carleman weights as follows. Let $\Gamma$ be any open subset of $\partial \Omega$, and let $\Gamma_0 = \partial \Omega \setminus \Gamma$ be the inaccessible part of the boundary. The proof of \cite{IUY} begins by showing that one can find a family of harmonic functions (limiting Carleman weights in 2D) $\{\phi_p\}_{p \in S}$, where $S$ is a dense subset of $\Omega$, such that 
$$
\Gamma_0 \subset \bigcap_{p \in S} \partial \Omega_{\rm tan}(\phi_p)
$$
and each $\phi_p$ is a Morse function (its critical points are nondegenerate) having a critical point at $p$. We refer to \cite{IUY} and the survey \cite{GT_survey} for more details about partial data problems in two dimensions.

\bigskip

\noindent 2.4. {\bf The result of \cite{DKSjU09}.}
As it is explained in Lemma \ref{prop_integralidentity} below, if $\Gamma_D = \Gamma_N = \Gamma$, where $\Gamma$ is a nonempty open subset of $\partial \Omega$, and if $C_{q_1}^{\Gamma,\Gamma} = C_{q_2}^{\Gamma,\Gamma}$, then 
$$
\int_{\Omega} (q_1-q_2) u_1 u_2 \,dx = 0
$$
for any $u_j \in H_{\Delta}(\Omega)$ with $(-\Delta+q_j)u_j = 0$ in $\Omega$ and $\supp(u_j|_{\partial \Omega}) \subset \Gamma$. In \cite{DKSjU09}, the authors consider a linearization of this assumption at $q_1=q_2=0$. The main result of \cite{DKSjU09} is:

\begin{thm} \label{thm_linearized}
Let $\Omega$ be a bounded connected open subset of $\mR^n$, $n \geq 2$, with connected $C^{\infty}$ boundary. The set of products of harmonic functions in $C^{\infty}(\overline{\Omega})$ which vanish on a closed proper subset $F \subset \partial \Omega$ is dense in $L^1(\Omega)$.
\end{thm}

While the proof of this theorem is quite different from the ones of the other results mentioned here, it does depend on the linear limiting Carleman weights $\phi(x) = \alpha \cdot x$ and the associated complex geometrical optics solutions $e^{\zeta \cdot x}$, where $\zeta \in \mC^n$ satisfies $\zeta \cdot \zeta = 0$, or equivalently $\zeta = \tau(\alpha+i\beta)$ where $\alpha, \beta \in \mR^n$ are unit vectors satisfying $\alpha \cdot \beta = 0$. See Section \ref{sec_linearized} for a sketch of the proof of this result.

\bigskip

\noindent 2.5. {\bf The results of \cite{KSa}.}
The partial data results in \cite{KSa} are valid on a class of Riemannian manifolds of dimension $n \geq 3$, but here we will only mention some consequences for domains in $\mR^3$. In all of the results below $\Omega \subset \mR^3$ is a bounded open connected set with $C^{\infty}$ boundary, and we consider the decomposition 
$$
\partial \Omega = \partial \Omega_+(\phi) \cup \partial \Omega_-(\phi) \cup \partial \Omega_{\rm tan}(\phi)
$$
with respect to a fixed limiting Carleman weight $\phi$. We will also assume that $\partial \Omega_{\rm tan}$ is decomposed as 
$$
\partial \Omega_{\rm tan} = \Gamma_a \cup \Gamma_i
$$
where $\Gamma_a$ is an open subset of $\partial \Omega_{\rm tan}$ that is accessible to boundary measurements, and $\Gamma_i$ is the inaccessible part. Further, we will assume that $\Gamma_D$ and $\Gamma_N$ are nonempty open sets in $\partial \Omega$ with $\partial \Omega_- \cup \Gamma_a \subset \Gamma_D$ and $\partial \Omega_+ \cup \Gamma_a \subset \Gamma_N$.

The content of the results below is that if the inaccessible part $\Gamma_i$ satisfies a geometric condition, then Dirichlet measurements on $\partial \Omega_- \cup \Gamma_a$ and Neumann measurements on $\partial \Omega_+ \cup \Gamma_a$ are sufficient to determine the unknown coefficient locally away from $\Gamma_i$.

The first theorem corresponds to the case $\phi(x) = x_1$ and the case where the inaccessible part $\Gamma_i$ is part of a cylindrical set. This result generalizes Theorem \ref{thm_ksu2} (if one chooses $E = \partial \Omega_0$ and $\Gamma_i = \emptyset$) and Theorem \ref{thm_i1} (if one chooses $\Omega_0$ and $E$ so that  $\Omega \cap \{ x_3 = 0 \} \subset \mR \times (\partial \Omega_0 \setminus E)$).

\begin{thm} \label{thm_ks1}
Suppose that $\Omega \subset \mR \times \Omega_0$ where $\Omega_0$ is a bounded open set with $C^{\infty}$ boundary in $\mR^2$. Let $\phi(x) = x_1$, and let $E$ be an open subset of $\partial \Omega_0$ such that $\Gamma_i$ satisfies 
$$
\Gamma_i \subset \mR \times (\partial \Omega_0 \setminus E).
$$
If $q_1, q_2 \in C(\overline{\Omega})$ and if 
$$
C_{q_1}^{\Gamma_D,\Gamma_N} = C_{q_2}^{\Gamma_D,\Gamma_N},
$$
then $q_1 = q_2$ in $\overline{\Omega} \cap (\mR \times O)$ where $O$ is the intersection of $\overline{\Omega}_0$ with the union of all lines in $\mR^2$ that have $\partial \Omega_0 \setminus E$ on one side.
\end{thm}

Note that if $\Omega_0$ is convex and $E$ is connected, then the set $O$ in the previous theorem is just the convex hull of $E$ in $\overline{\Omega}_0$ and one recovers the potential in $\overline{\Omega} \cap (\mR \times {\rm ch}_{\overline{\Omega}_0}(E))$. The next theorem is related to the logarithmic weight and generalizes Theorem \ref{thm_ksu1} (choose $E = \partial \Omega_0$ and $\Gamma_i = \emptyset$). The inaccessible part of the boundary must now be part of a conical set.

\begin{thm} \label{thm_ks2}
Let $\overline{\Omega} \subset \{ x_3 > 0 \}$ and let $\phi(x) = \log \,\abs{x}$. Consider the hemisphere $S^2_+ = \{ \omega \in S^2 \,;\, \omega_3 > 0 \}$, and let $M_0$ be a compact subdomain of $S^2_+$ with $C^{\infty}$ boundary. Let $E$ be an open subset of $\partial M_0$ such that $\Gamma_i$ satisfies 
$$
\Gamma_i \subset \{ r\omega \,;\, r > 0, \omega \in \partial M_0 \setminus E \}.
$$
If $q_1, q_2 \in C(\overline{\Omega})$ and if 
$$
C_{q_1}^{\Gamma_D,\Gamma_N} = C_{q_2}^{\Gamma_D,\Gamma_N},
$$
then $q_1 = q_2 \text{ in } \overline{\Omega} \cap \{ r \omega \,;\, r > 0, \omega \in O \}$ where $O$ is the union of all great circle segments in $S^2_+$ such that $\partial M_0 \setminus E$ is on one side of the hyperplane containing the great circle segment.
\end{thm}

The final theorem involves the weight $\phi(x) = {\rm arg}(x_1+i x_2)$ and does not have a counterpart in the previous results. It corresponds to a case where the inaccessible part of the boundary is part of a surface of revolution.

\begin{thm} \label{thm_ks3}
Let $\Omega \subset \{ x_1 > 0 \}$ and let $\phi(x) = {\rm arg}(x_1+i x_2)$. 
Let $S = \{ (x_1,0,x_3) \,;\, x_1 > 0 \}$ be a half-plane with hyperbolic geodesics given by the half-circles (with $R > 0$ and $\alpha \in \mR$) 
$$
(x_1(t), x_2(t), x_3(t)) = (R \sin t, 0, R \cos t + \alpha), \quad t \in (0,\pi)
$$
or the lines 
$$
(x_1(t), x_2(t), x_3(t)) = (t, 0, \alpha), \quad t > 0.
$$
Let $M_0$ be a compact subdomain of $S$ with smooth boundary, and let $E$ be an open subset of $\partial M_0$ such that $\Gamma_i$ satisfies 
$$
\Gamma_i \subset \{ R_{\theta}(\partial M_0 \setminus E) \,;\, \theta \in (-\pi, \pi) \}
$$
where $R_{\theta}x = (\tilde{R}_{\theta}(x_1,x_2), x_3)$ and $\tilde{R}_{\theta}$ rotates vectors in $\mR^2$ by angle $\theta$ counterclockwise. That is, we assume that the inaccessible part $\Gamma_i$ is contained in a surface of revolution obtained by rotating the boundary curve $\partial M_0 \setminus E$ about the $x_3$-axis.

If $q_1, q_2 \in C(\overline{\Omega})$ and if 
$$
C_{q_1}^{\Gamma_D,\Gamma_N} = C_{q_2}^{\Gamma_D,\Gamma_N},
$$
then $q_1 = q_2 \text{ in } \overline{\Omega} \cap \{ R_{\theta}(O) \,;\, \theta \in (-\pi,\pi) \}$ where $O$ is the union of all geodesics in $S$ that have $\partial M_0 \setminus E$ on one side.
\end{thm}

Note that the above results are local results that allow to determine the unknown potential locally away from the inaccessible part $\Gamma_i$. The paper \cite{KSa} also contains results where one obtains information on the potential near the inaccessible part $\Gamma_i$, but this information comes in the form of integrals along broken geodesic rays in the transversal manifold. In the Euclidean case, a continuous curve $\gamma: [0,L] \to \overline{\Omega}_0$ is called a \emph{nontangential broken ray} if $\gamma$ is obtained by following straight lines that are reflected in the standard way (angle of incidence = angle of reflection) whenever they hit $\partial \Omega_0$, all reflections are nontangential, and all reflection points are distinct. One also needs a somewhat stronger assumption on the sets $\Gamma_D$ and $\Gamma_N$ (this assumption was made for simplicity and could be removed in many cases).

Let us state a counterpart of Theorem \ref{thm_ks1} which involves the broken ray transform with exponential attenuation. Note that in the absence of the stronger assumption on $\Gamma_D$ and $\Gamma_N$, this result would generalize Theorem \ref{thm_ks1}.

\begin{thm} \label{thm_ks4}
Suppose that $\Omega \subset \mR \times \Omega_0$ where $\Omega_0$ is a bounded open set with $C^{\infty}$ boundary in $\mR^2$. Let $\phi(x) = x_1$, and let $E$ be an open subset of $\partial \Omega_0$ such that $\Gamma_i$ satisfies 
$$
\Gamma_i \subset \mR \times (\partial \Omega_0 \setminus E).
$$
Let $q_1, q_2 \in C(\overline{\Omega})$, let $\Gamma$ be a neighborhood of $\overline{\partial \Omega_+ \cup \partial \Omega_- \cup \Gamma_a}$, and assume that 
$$
C_{q_1}^{\Gamma,\Gamma} = C_{q_2}^{\Gamma,\Gamma}.
$$
Then for any nontangential broken ray $\gamma: [0,L] \to \overline{\Omega}_0$ with endpoints on $E$, and given any real number $\lambda$, one has 
$$
\int_0^L e^{-2\lambda t} (q_1-q_2)\ehat(2\lambda,\gamma(t)) \,dt = 0.
$$
Here $(\,\cdot\,)\ehat$ is the Fourier transform in the $x_1$ variable, and $q_1-q_2$ is extended by zero to $\mR^3 \setminus \overline{\Omega}$.
\end{thm}

It follows from this result that one could recover the unknown potential also near $\Gamma_i$ if one knew how to invert the broken ray transform, that is, to recover a function in $\overline{\Omega}_0$ from its (attenuated) integrals over nontangential broken rays with endpoints on a given open subset $E \subset \partial \Omega_0$. This is a question of independent interest and there seem to be only partial results in this direction, see \cite{Eskin_obstacles}, \cite{Ilmavirta}.

\section{Strategy of proof} \label{sec_strategy}

In this section we give an outline of the proof of the partial data results described above. The proofs proceed in three steps:
\begin{enumerate}
\item[1.]
Via an integral identity, the partial data uniqueness question is reduced to showing that products of solutions of Schr\"odinger equations, which vanish on suitable parts of the boundary, are dense in $L^1(\Omega)$.
\item [2.]
Construction of a family of special complex geometrical optics solutions to Schr\"odinger equations, and showing that certain transforms of any function that is $L^2$-orthogonal to products of complex geometrical optics solutions must vanish.
\item[3.] 
Showing that the transforms arising in Step 2 are injective, which completes the proof that products of solutions are dense and also the uniqueness proof.
\end{enumerate}
This general outline roughly applies to all of the results in Section \ref{sec_results}, but the particulars (the choice of complex geometrical optics solutions, and the transforms in the final step) vary from case to case.

\bigskip

\noindent 3.1. {\bf Reduction to the density of products of solutions.}
This step is achieved by the following basic integral identity.

\begin{lemma} \label{prop_integralidentity}
If $\Gamma_D, \Gamma_N \subset \partial \Omega$ are open and if $C_{q_1}^{\Gamma_D, \Gamma_N} = C_{q_2}^{\Gamma_D, \Gamma_N}$, then 
$$
\int_\Omega (q_1-q_2) u_1 u_2 \,dx = 0
$$
for any $u_j \in H_{\Delta}(\Omega)$ satisfying $(-\Delta + q_j) u_j = 0$ in $\Omega$ and 
$$
\supp(u_1|_{\partial \Omega}) \subset \Gamma_D, \quad \supp(u_2|_{\partial \Omega}) \subset \Gamma_N.
$$
\end{lemma}
\begin{proof}
Let $u_j$ be as stated. Since $C_{q_1}^{\Gamma_D, \Gamma_N} = C_{q_2}^{\Gamma_D, \Gamma_N}$, there is a function $\tilde{u}_2 \in H_{\Delta}(\Omega)$ with $(-\Delta + q_2) \tilde{u}_2 = 0$ in $\Omega$, $\supp(\tilde{u}_2|_{\partial \Omega}) \subset \Gamma_D$, and 
$$
(u_1|_{\Gamma_D}, \partial_{\nu} u_1|_{\Gamma_N}) = (\tilde{u}_2|_{\Gamma_D}, \partial_{\nu} \tilde{u}_2|_{\Gamma_N}).
$$
Using that $u_1$, $u_2$ and $\tilde{u}_2$ are solutions, we have 
\begin{align*}
\int_\Omega (q_1-q_2) u_1 u_2 \,dx &= \int_\Omega \left[ (\Delta u_1) u_2 - u_1 (\Delta u_2) \right] \,dx \\
 &= \int_\Omega \left[ (\Delta (u_1-\tilde{u}_2)) u_2 - (u_1-\tilde{u}_2) (\Delta u_2) \right] \,dx.
\end{align*}
Now $u_1-\tilde{u}_2|_{\partial \Omega} = 0$, so in fact $u_1-\tilde{u}_2 \in H^2(\Omega)$ by the properties of the space $H_{\Delta}(\Omega)$, see \cite{BU}. Recall also that $C^{\infty}(\overline{\Omega})$ is dense in $H_{\Delta}(\Omega)$ and that $u_2|_{\partial \Omega} \in H^{-1/2}(\partial \Omega)$ and $\partial_{\nu} u_2|_{\partial \Omega} \in H^{-3/2}(\partial \Omega)$. These facts make it possible to integrate by parts, and we obtain that 
$$
\int_\Omega (q_1-q_2) u_1 u_2 \,dx = \int_{\partial \Omega} \left[ (\partial_{\nu} (u_1-\tilde{u}_2)) u_2 - (u_1-\tilde{u}_2) (\partial_{\nu} u_2) \right] \,dS
$$
in the weak sense. The last expression vanishes since $\partial_{\nu} (u_1-\tilde{u}_2)|_{\Gamma_N} = 0$ and $\supp(u_2|_{\partial \Omega}) \subset \Gamma_N$.
\end{proof}

The next statement makes precise the reduction of the partial data problem to density of products.

\begin{clry} \label{corollary_reduction_density}
Let $\Omega \subset \mR^n$ be a bounded open set with $C^{\infty}$ boundary, let $q_1, q_2 \in L^{\infty}(\Omega)$, and let $\Gamma_D$ and $\Gamma_N$ be nonempty open subsets of $\partial \Omega$. If the set 
\begin{align*}
\{ u_1 u_2 \ ;\  &u_j \in H_{\Delta}(\Omega), \ \ (-\Delta + q_j) u_j = 0 \text{ in }\Omega, \\
 & \supp(u_1|_{\partial \Omega}) \subset \Gamma_D, \ \ \supp(u_2|_{\partial \Omega}) \subset \Gamma_N \}
\end{align*}
is dense in $L^1(\Omega)$, then the condition $C_{q_1}^{\Gamma_D, \Gamma_N} = C_{q_2}^{\Gamma_D, \Gamma_N}$ implies $q_1 = q_2$.
\end{clry}

\bigskip

\noindent 3.2. {\bf Complex geometrical optics solutions.}
To show the density claim in Corollary \ref{corollary_reduction_density}, one constructs special solutions $u_j \in H_{\Delta}(\Omega)$ of the equations $(-\Delta+q_j)u_j = 0$ in $\Omega$ with the required boundary conditions. This will involve complex geometrical optics solutions having the form 
\begin{align*}
u_1 &= e^{-\tau \phi}(m_1 + r_1), \\
u_2 &= e^{\tau \phi}(m_2 + r_2)
\end{align*}
where $\tau > 0$ is a large parameter, $\varphi$ is a real valued weight function, $m_1$ and $m_2$ are complex amplitudes, and $r_1$ and $r_2$ are correction terms that are small when $\tau$ is large. The solutions $u_1$ and $u_2$ are chosen with opposite signs for the exponential ($e^{-\tau \phi}$ and $e^{\tau \phi}$) to make sure that these exponentials go away in the product, allowing one to consider the limit of $u_1 u_2$ as $\tau \to \infty$.

The point is to choose the amplitudes $m_1$, $m_2$ so that $\Delta(e^{-\tau \phi} m_1) \approx 0$ and $\Delta(e^{\tau \phi} m_2) \approx 0$ in a suitable sense, ensuring that the approximate solutions $e^{-\tau \phi} m_1$ and $e^{\tau \phi} m_2$ can be corrected to exact solutions of $(-\Delta+q_j) u_j = 0$. The complex amplitudes often involve phase factors $e^{i\tau \psi_j}$, so the solutions contain complex exponentials $e^{\tau (-\phi+i\psi_1)}$ and $e^{\tau(\phi + i\psi_2)}$. This explains why these are called complex geometrical optics solutions.

The choice of the weight function $\phi$ is crucial in this process. The correction terms $r_1$ and $r_2$ are obtained by solving the equations 
\begin{align*}
e^{\tau \phi}(-\Delta+q_1)(e^{-\tau \phi} r_1) &= f_1 \quad \text{in } \Omega, \\
e^{-\tau \phi}(-\Delta+q_2)(e^{\tau \phi} r_2) &= f_2 \quad \text{in } \Omega
\end{align*}
where 
\begin{align*}
f_1 &= -e^{\tau \phi}(-\Delta+q_1)(e^{-\tau \phi} m_1), \\
f_2 &= -e^{-\tau \phi}(-\Delta+q_2)(e^{\tau \phi} m_2).
\end{align*}
Thus one would like to have solvability results for the conjugated operators $e^{\tau \phi}(-\Delta+q_1)e^{-\tau \phi}$ and $e^{-\tau \phi}(-\Delta+q_2)e^{\tau \phi}$ that come with estimates such as 
$$
\norm{r_j}_{L^2(\Omega)} \leq C \tau^{-1} \norm{f_j}_{L^2(\Omega)},
$$
for a constant $C$ uniform over all sufficiently large $\tau$. Additionally, one needs some control of the boundary values of $r_1$ and $r_2$ on parts of the boundary.

The above procedure can be carried out if $\phi$ is a \emph{limiting Carleman weight}. We refer to \cite{KSU}, \cite{DKSaU} for more information about the microlocal condition characterizing these weights that ensures the right kind of solvability results. However, as discussed in Section \ref{sec_results}, limiting Carleman weights are just harmonic functions (with some restriction on their critical points) if $n=2$, and in $\mR^n$ with $n \geq 3$ there are only six possibilities up to translation and scaling. Moreover, as will be explained in Section \ref{sec_carleman} below, one can expect relatively good control of the boundary values of $r_1$ on $\partial \Omega_+(\phi)$ and weak control on $\partial \Omega_{\rm tan}(\phi)$, and similarly relatively good control of the boundary values of $r_2$ on $\partial \Omega_-(\phi)$ and weak control on $\partial \Omega_{\rm tan}(\phi)$ (this is due to the sign in the exponential $e^{\mp \tau \phi}$). This will follow from a Carleman estimate (a weighted $L^2$ estimate with exponential weights depending on the large parameter $\tau$) with boundary terms.

It remains to describe the construction of the amplitudes $m_1$ and $m_2$. The following three conditions are typical ones that one might like the amplitudes to satisfy:
\begin{enumerate}
\item[1.]
$\norm{e^{\tau \phi} (-\Delta+q_1)(e^{-\tau \phi} m_1)}_{L^2(\Omega)} = \norm{e^{-\tau \phi} (-\Delta+q_2)(e^{\tau \phi} m_2)}_{L^2(\Omega)} = O(1)$ as $\tau \to \infty$,
\item[2.]
$\norm{m_1}_{L^2(\partial \Omega_+(\phi))} = \norm{m_2}_{L^2(\partial \Omega_-(\phi))} = O(1)$, and $\norm{m_j}_{L^2(\partial \Omega_{\rm tan}(\phi))} = o(1)$ as $\tau \to \infty$, 
\item[3.]
the set of limits $\{ \lim_{\tau \to \infty} m_1 m_2 \}$ for all such $m_1$ and $m_2$ is a dense set in $L^1(\Omega)$.
\end{enumerate}
The first condition means that $e^{-\tau \phi} m_1$ and $e^{\tau \phi} m_2$ are approximate solutions in a suitable sense, allowing to find correction terms $r_j$ with $\norm{r_j}_{L^2(\Omega)} = o(1)$ as $\tau \to \infty$. The second condition comes from the fact that one wants roughly that 
$$
u_1|_{\partial \Omega_+ \cup \partial \Omega_{\rm tan}} = 0, \qquad u_2|_{\partial \Omega_- \cup \partial \Omega_{\rm tan}} = 0,
$$
or equivalently 
$$
r_1|_{\partial \Omega_+ \cup \partial \Omega_{\rm tan}} = -m_1|_{\partial \Omega_+ \cup \partial \Omega_{\rm tan}}, \qquad r_2|_{\partial \Omega_- \cup \partial \Omega_{\rm tan}} = -m_2|_{\partial \Omega_- \cup \partial \Omega_{\rm tan}}.
$$
Since the Carleman estimate with boundary terms will give relatively good control of the boundary values of $r_j$ on $\partial \Omega_+$ or $\partial \Omega_-$ and weak control on $\partial \Omega_{\rm tan}$, it is enough that $m_j$ are bounded on $\partial \Omega_+$ or $\partial \Omega_-$ but they need to be small as $\tau \to \infty$ (or vanish completely) on $\partial \Omega_{\rm tan}$. The third condition follows since 
$$
\lim_{\tau \to \infty} u_1 u_2 = \lim_{\tau \to \infty} m_1 m_2.
$$
Thus, to show density of the products $u_1 u_2$, it is enough to show density of the set $\{ \lim_{\tau \to \infty} m_1 m_2 \}$ for all admissible $m_1$ and $m_2$. This leads to the transforms considered in the final step.

\bigskip

\noindent 3.3. {\bf Injectivity of transforms.}
In the density of products approach, one needs to show that any function $f$ in, say, $C(\overline{\Omega})$, for which 
$$
\int_{\Omega} f u_1 u_2 \,dx = 0
$$
for all solutions $u_j \in H_{\Delta}(\Omega)$ of $(-\Delta+q_j)u_j = 0$ in $\Omega$, must satisfy $f = 0$. By the discussion in this section, one can try to do this by constructing complex geometrical optics solutions with amplitudes $m_j = m_j(\,\cdot\,;\tau,\alpha)$ depending on the large parameter $\tau$ and some additional parameters described by $\alpha$ which ranges over some parameter set $A$. Assume that for each $\alpha$, there is a bounded measure $\mu_{\alpha}$ on $\Omega$ such that 
$$
\lim_{\tau \to \infty} m_1(\,\cdot\,;\tau,\alpha) m_2(\,\cdot\,;\tau,\alpha) = \mu_{\alpha}
$$
in the weak topology of measures in $\Omega$. The orthogonality condition then implies that 
$$
\int_{\Omega} f \,d\mu_{\alpha} = 0, \quad \alpha \in A.
$$
That is, the transform $Tf$ of $f$ vanishes, where 
$$
Tf(\alpha) = \int_{\Omega} f \,d\mu_{\alpha}, \quad \alpha \in A.
$$
If this transform is injective, then $f=0$ and the density of products follows.

Many different transforms have appeared in this connection. In the full data result of \cite{SU}, the transform $T$ was just the usual Fourier transform. In the partial data result \cite{BU} one obtained the Fourier transform in a cone, and also Isakov's approach \cite{I} results in the Fourier transform. The partial data result \cite{KSU} involves a more complicated transform whose injectivity was proved by analytic microlocal analysis, and this argument was simplified in \cite{DKSjU07} where matters reduce to inverting the two-plane Radon transform. In the two-dimensional case, the transform appearing \cite{Bu} and \cite{IUY} is related to stationary phase. Finally, in \cite{DKSaU}, \cite{KSa} a mixed transform involving the Fourier transform in one direction and X-ray transforms (or more generally attenuated geodesic ray transforms) in other directions was employed.

\section{Carleman estimates} \label{sec_carleman}

As discussed in Section \ref{sec_strategy}, the first step in the construction of complex geometrical optics solutions with controlled boundary behaviour is a Carleman estimate with boundary terms. Here we give the proof of such an estimate for the conjugated operator $e^{\phi/h}(-\Delta + q)e^{-\phi/h}$ in $\Omega$, where the limiting Carleman weight $\phi$ is a linear function and $h > 0$ is a small parameter (in the previous notation, $h = 1/\tau$). Following \cite{KSU}, it is useful to consider a slightly modified weight 
$$
\phi_{\eps} = \phi + h f_{\eps}
$$
where $f_{\eps}$ is a smooth real valued function in $\overline{\Omega}$ depending on a small parameter $\eps$ but independent of $h$. The convexity of $f_{\eps}$ will lead to improved lower bounds in terms of the $L^2(\Omega)$ norms of $u$ and $\nabla u$. On the other hand, the sign of $\partial_{\nu} \phi_{\eps}$ in the boundary term of the Carleman estimate will allow to control functions on different parts of the boundary. Of special interest is the set $\partial \Omega_{\rm tan}$ where $\partial_{\nu} \phi = 0$, so one has 
$$
\partial_{\nu} \phi_{\eps} = h \partial_{\nu} f_{\eps} \quad \text{on } \partial \Omega_{\rm tan}.
$$
We would like to have $\partial_{\nu} f_{\eps} < 0$ on $\partial \Omega_{\rm tan}$. It is not easy to find a global convex function $f_{\eps}$ satisfying the last condition for a general set $\partial \Omega_{\rm tan}$. However, splitting $f_{\eps}$ to a convex part whose normal derivative vanishes on $\partial \Omega_{\rm tan}$ and another part which ensures the correct sign on $\partial \Omega_{\rm tan}$ will give the required result. We will use the notation $D = -i\nabla$, $(u,v) = (u,v)_{L^2(\Omega)}$, $\norm{u} = \norm{u}_{L^2(\Omega)}$, and also $L^2$ inner products and norms on parts of $\partial \Omega$ are indicated by a subscript.

\begin{prop} \label{prop_carleman_first}
Let $\Omega \subset \mR^n$, $n \geq 2$, be a bounded open set with $C^{\infty}$ boundary, let $\varphi(x) = \alpha \cdot x$ where $\alpha \in \mR^n$ is a unit vector, and let $\kappa$ be a smooth real valued function in $\overline{\Omega}$ so that $\partial_{\nu} \kappa = -1$ on $\partial \Omega$. Let also $q \in L^{\infty}(\Omega)$. There are constants $\eps, C_0, h_0 > 0$ with $h_0 \leq \eps \leq 1$ such that for the weight 
$$
\varphi_{\eps} = \varphi + \frac{h}{\eps} \frac{\varphi^2}{2} + h \kappa
$$
where $0 < h \leq h_0$, one has 
\begin{multline*}
\frac{h^3}{C_0} (\abs{\partial_{\nu} \phi_{\eps}} \partial_{\nu} u,\partial_{\nu} u)_{\partial \Omega_{-}(\phi_{\eps})} + \frac{h^2}{C_0} (\norm{u}^2 + \norm{hDu}^2) \\
 \leq \norm{e^{\phi/h}((hD)^2 + h^2 q) (e^{-\phi/h} u)}^2 + h^3 ( \abs{\partial_{\nu} \phi_{\eps}} \partial_{\nu} u,\partial_{\nu} u)_{\partial \Omega_+(\phi_{\eps})}
\end{multline*}
for any $u \in C^{\infty}(\overline{\Omega})$ with $u|_{\partial \Omega} = 0$.
\end{prop}
\begin{proof}
It is sufficient to consider the case $\alpha = e_1$, so $\phi(x) = x_1$ and
$$
\phi_{\eps} = x_1 + \frac{h}{\eps} \frac{x_1^2}{2} + h \kappa.
$$
Let $P_{\phi_{\eps}} = e^{\phi_{\eps}/h} (hD)^2 e^{-\phi_{\eps}/h}$, and write 
$$
P_{\phi_{\eps}} = A + iB
$$
where $A$ and $B$ are the formally self-adjoint operators 
\begin{gather*}
A = (hD)^2 - \abs{\nabla \phi_{\eps}}^2, \\
B = \nabla \phi_{\eps} \circ hD + hD \circ \nabla \phi_{\eps} = 2\nabla \phi_{\eps} \cdot hD + \frac{h}{i} \Delta \phi_{\eps}.
\end{gather*}
Then, using the condition $u|_{\partial \Omega} = 0$,  
\begin{align*}
\norm{P_{\phi_{\eps}} u}^2 &= ((A+iB)u, (A+iB)u) \\
 &= \norm{Au}^2 + \norm{Bu}^2 + i(Bu,Au) - i(Au,Bu) \\
 &= \norm{Au}^2 + \norm{Bu}^2 + (i[A,B]u,u) - i h^2 (Bu,\partial_{\nu} u)_{\partial \Omega} \\
 &= \norm{Au}^2 + \norm{Bu}^2 + (i[A,B]u,u) - 2 h^3 ((\partial_{\nu} \phi_{\eps}) \partial_{\nu} u,\partial_{\nu} u)_{\partial \Omega}.
\end{align*}
We compute the commutator 
\begin{align*}
 &i[A,B]u = h \big[ ((hD)^2 - \abs{\nabla \phi_{\eps}}^2)(2\nabla \phi_{\eps} \cdot \nabla u + (\Delta \phi_{\eps})u) \\
 &\quad \quad \quad - (2\nabla \phi_{\eps} \cdot \nabla + \Delta \phi_{\eps})((hD)^2 u - \abs{\nabla \phi_{\eps}}^2 u) \big] \\
 &= h \big[ 2 \nabla (hD)^2 \phi_{\eps} \cdot \nabla u + 4 hD \sum_{k=1}^n \partial_k \phi_{\eps} \cdot hD \partial_k u + ((hD)^2 \Delta \phi_{\eps}) u \\
  &\quad \quad \quad + 2 hD \Delta \phi_{\eps} \cdot hDu + 2 \nabla \phi_{\eps} \cdot \nabla(\abs{\nabla \phi_{\eps}}^2) u \big] \\
  &= h \big[ 4 (\phi_{\eps}'' \nabla \phi_{\eps} \cdot \nabla \phi_{\eps}) u - 4 h^2 \sum_{j,k=1}^n \partial_{jk} \phi_{\eps} \partial_{jk} u - 4 h^2 \nabla \Delta \phi_{\eps} \cdot \nabla u \\
  &\quad \quad \quad - h^2 (\Delta^2 \phi_{\eps}) u \big].
\end{align*}
Here $\phi_{\eps}''$ is the Hessian matrix of $\phi_{\eps}$. Integrating by parts once, using again that $u|_{\partial \Omega} = 0$, yields 
\begin{multline*}
(i[A,B]u, u) = 4h^3(\phi_{\eps}''\nabla u, \nabla u) + 4h((\phi_{\eps}'' \nabla \phi_{\eps} \cdot \nabla \phi_{\eps}) u, u) \\
 - h^3 ((\Delta^2 \phi_{\eps})u, u).
\end{multline*}
Now we use that $\varphi$ is a linear function, so that $\phi_{\eps}'' = \frac{h}{\eps} e_1 \otimes e_1 + h \kappa''$. Combining this with the previous computation for $\norm{P_{\phi_{\eps}} u}^2$, we obtain 
\begin{multline*}
\norm{P_{\phi_{\eps}} u}^2 = \norm{Au}^2 + \norm{Bu}^2 + \frac{4 h^4}{\eps} \norm{\partial_1 u}^2 + 4 h^4 (\kappa'' \nabla u, \nabla u) \\
 + \frac{4 h^2}{\eps} \norm{(\partial_1 \phi_{\eps}) u}^2  + 4 h^2 ((\kappa'' \nabla \phi_{\eps} \cdot \nabla \phi_{\eps}) u, u) - h^4 ((\Delta^2 \kappa)u, u) \\
 - 2 h^3 ((\partial_{\nu} \phi_{\eps}) \partial_{\nu} u,\partial_{\nu} u)_{\partial \Omega}.
\end{multline*}

We have $\nabla \phi_{\eps} = \nabla \phi + h(\frac{\phi \nabla \phi}{\eps} + \nabla \kappa)$. Assume that $h_0$ is so small that 
$$
\left\lvert h \left( \frac{\phi \nabla \phi}{\eps} + \nabla \kappa \right) \right\rvert \leq 1/2 \text{ on $\overline{\Omega}$ when $0 < h < h_0$.}
$$
Then $1/2 \leq \abs{\nabla \phi_{\eps}} \leq 3/2$ on $\overline{\Omega}$, and for some constant $C_{\kappa} \geq 1$ 
\begin{multline*}
\norm{P_{\phi_{\eps}} u}^2 \geq \norm{Au}^2 + \norm{Bu}^2 + \frac{4 h^2}{\eps} \norm{h D_1 u}^2 - C_{\kappa} h^2 \norm{hDu}^2 \\
 + \frac{h^2}{\eps} \norm{u}^2  - C_{\kappa} h^2 \norm{u}^2 - C_{\kappa} h^4 \norm{u}^2 - 2 h^3 ((\partial_{\nu} \phi_{\eps}) \partial_{\nu} u,\partial_{\nu} u)_{\partial \Omega}.
\end{multline*}
We obtain a further lower bound for $\norm{hDu}^2$ from the term $\norm{Au}^2$: given any $M > 0$, and using that $u|_{\partial \Omega} = 0$, one has 
\begin{align*}
h^2 \norm{hDu}^2 &= h^2 ((hD)^2 u, u) = h^2 (Au, u) + h^2 (\abs{\nabla \phi_{\eps}}^2 u, u) \\
 &\leq \frac{1}{2M} \norm{Au}^2 + \frac{M h^4}{2} \norm{u}^2 + 4h^2 \norm{u}^2.
\end{align*}
From this lower bound for $\norm{Au}^2$ and the trivial estimate $\norm{Bu}^2 \geq 0$, it follows that 
\begin{multline*}
\norm{P_{\phi_{\eps}} u}^2 \geq 2M h^2 \norm{hDu}^2 - M^2 h^4 \norm{u}^2 - 8M h^2 \norm{u}^2 + \frac{4 h^2}{\eps} \norm{h D_1 u}^2 \\
 - C_{\kappa} h^2 \norm{hDu}^2 + \frac{h^2}{\eps} \norm{u}^2  - C_{\kappa} h^2 \norm{u}^2 - C_{\kappa} h^4 \norm{u}^2 \\
 - 2 h^3 ((\partial_{\nu} \phi_{\eps}) \partial_{\nu} u,\partial_{\nu} u)_{\partial \Omega}.
\end{multline*}
Choosing $M = (16\eps)^{-1}$, choosing $\eps$ sufficiently small depending on $C_{\kappa}$ and then choosing $h_0$ sufficiently small depending on $\eps$ implies that 
$$
\norm{P_{\phi_{\eps}} u}^2 \geq \frac{h^2}{16 \eps} (\norm{u}^2 + \norm{hDu}^2) - 2 h^3 ((\partial_{\nu} \phi_{\eps}) \partial_{\nu} u,\partial_{\nu} u)_{\partial \Omega}
$$
when $0 < h < h_0$. By further decreasing $h_0$ we may add a potential $q \in L^{\infty}(\Omega)$ to get the estimate 
$$
2 \norm{(P_{\phi_{\eps}} + h^2 q)u}^2 \geq \frac{h^2}{100 \eps} (\norm{u}^2 + \norm{hDu}^2) - 2 h^3 ((\partial_{\nu} \phi_{\eps}) \partial_{\nu} u,\partial_{\nu} u)_{\partial \Omega}
$$
Finally, we replace $u$ by $e^{x_1^2/2\eps + \kappa} u$, where $u \in C^{\infty}(\overline{\Omega})$ and $u|_{\partial \Omega} = 0$, and use the fact that 
$$
1/C \leq e^{x_1^2/2\eps + \kappa} \leq C \ \ \text{on } \overline{\Omega}.
$$
The required estimate follows.
\end{proof}

We now pass from $\phi_{\eps}$ to $\phi$ in the boundary terms of the previous result, making use of the special properties of $\phi_{\eps}$ on $\partial \Omega$. Note that the factor $h^4$ in the boundary term on $\{x \in \partial \Omega \,;\, -\delta < \partial_{\nu} \phi(x) < h/3 \}$ below is weaker than the factor $h^3$ in the other boundary terms. This follows from the fact that $\partial_{\nu} \phi_{\eps} = h \partial_{\nu} \kappa = -h$ in the set where $\partial_{\nu} \phi$ vanishes, so one only has weak control near $\partial \Omega_{\rm tan}(\phi)$.

\begin{prop} \label{prop_carleman_second}
Let $q \in L^{\infty}(\Omega)$, and let $\phi(x) = \pm x_1$. There exist constants $C_0, h_0 > 0$ such that whenever $0 < h \leq h_0$ and $\delta > 0$, one has 
\begin{multline*}
\frac{\delta h^3}{C_0} \norm{\partial_{\nu} u}_{L^2(\{ \partial_{\nu} \phi \leq -\delta \})}^2 + \frac{h^4}{C_0} \norm{\partial_{\nu} u}_{L^2(\{ -\delta < \partial_{\nu} \phi < h/3 \})}^2+ \frac{h^2}{C_0} (\norm{u}^2 + \norm{h \nabla u}^2) \\
 \leq \norm{e^{\phi/h}(-h^2 \Delta + h^2 q) (e^{-\phi/h} u)}^2 + h^3 \norm{\partial_{\nu} u}_{L^2(\{ \partial_{\nu} \phi \geq h/3 \})}^2
\end{multline*}
for any $u \in C^{\infty}(\overline{\Omega})$ with $u|_{\partial \Omega} = 0$.
\end{prop}
\begin{proof}
Note that 
$$
\partial_{\nu} \phi_{\eps} = \left( 1 + \frac{h}{\eps} \phi \right) \partial_{\nu} \phi + h \partial_{\nu} \kappa = \left( 1 + \frac{h}{\eps} \phi \right) \partial_{\nu} \phi - h.
$$
We choose $h_0$ so small that whenever $h \leq h_0$, one has for $x \in \overline{\Omega}$ 
$$
\frac{1}{2} \leq 1 + \frac{h}{\eps} \phi(x) \leq \frac{3}{2}.
$$
On the set where $\partial_{\nu} \phi(x) \leq -\delta$, we have  
$$
\abs{\partial_{\nu} \phi_{\eps}} \geq \delta/2.
$$
If $-\delta < \partial_{\nu} \phi < h/3$, we use the estimate 
$$
\abs{\partial_{\nu} \phi_{\eps}} \geq h/2.
$$
Moreover, $\abs{\partial_{\nu} \phi_{\eps}} \leq C_0$ on $\partial \Omega$. Since $\{ \partial_{\nu} \phi < h/3 \} \subset \{ \partial_{\nu} \phi_{\eps} < 0 \}$ and $ \{ \partial_{\nu} \phi_{\eps} \geq 0 \} \subset \{ \partial_{\nu} \phi \geq h/3 \}$, the result follows from Proposition \ref{prop_carleman_first} after replacing $C_0$ by some larger constant.
\end{proof}

We can now obtain a solvability result from the previous Carleman estimate in a standard way by duality (see \cite{BU, KSU, NS}). There is a slight technical complication since the solution will be in $L^2$ but not in $H^1$. To remedy this we will work with the space 
$$
H_{\Delta}(\Omega) = \{ u \in L^2(\Omega) \,;\, \Delta u \in L^2(\Omega) \}
$$
with norm $\norm{u}_{H_{\Delta}} = \norm{u}_{L^2} + \norm{\Delta u}_{L^2}$. It is known (see \cite{BU}) that $H_{\Delta}(\Omega)$ is a Hilbert space having $C^{\infty}(\overline{\Omega})$ as a dense subset, and there is a well defined bounded trace operator from $H_{\Delta}(\Omega)$ to $H^{-1/2}(\partial \Omega)$ and a normal derivative operator from $H_{\Delta}(\Omega)$ to $H^{-3/2}(\partial \Omega)$. We also recall that if $u \in H_{\Delta}(\Omega)$ and $u|_{\partial \Omega} \in H^{3/2}(\partial \Omega)$, then $u \in H^2(\Omega)$.

\begin{prop} \label{prop_carleman_solvability}
Let $\Omega \subset \mR^n$ be a bounded open set with $C^{\infty}$ boundary, let $q \in L^{\infty}(\Omega)$, and let $\phi(x) = \pm x_1$. There exist constants $C_0, \tau_0 > 0$ such that when $\tau \geq \tau_0$ and $\delta > 0$, then for any $f \in L^2(\Omega)$ and $f_- \in L^2(S_- \cup S_0)$ there exists $u \in L^2(\Omega)$ satisfying $e^{\tau \phi} u \in H_{\Delta}(\Omega)$ and $e^{\tau \phi} u|_{\partial \Omega} \in L^2(\partial \Omega)$ such that  
$$
e^{-\tau \phi}(-\Delta + q)(e^{\tau \phi} u) = f \ \ \text{in $\Omega$}, \quad e^{\tau \phi} u|_{S_- \cup S_0} = e^{\tau \phi} f_-,
$$
and 
$$
\norm{u}_{L^2(M)} \leq C_0 (\tau^{-1} \norm{f}_{L^2(\Omega)} + (\delta \tau)^{-1/2} \norm{f_-|_{S_-}}_{L^2(S_-)} + \norm{f_-|_{S_0}}_{L^2(S_0)}).
$$
Here $S_{\pm}$ and $S_0$ are the following subsets of $\partial \Omega$:
$$
S_- = \{ \partial_{\nu} \phi \leq -\delta \}, S_0 = \{ -\delta < \partial_{\nu} \phi < 1/(3\tau) \}, S_+ = \{ \partial_{\nu} \phi \geq 1/(3\tau) \}.
$$
\end{prop}
\begin{proof}
Write $Lv = e^{\tau \phi}(-\Delta + \bar{q}) (e^{-\tau \phi} v)$ and $\tau = 1/h$, $\tau_0 = 1/h_0$. We rewrite the Carleman estimate of Proposition \ref{prop_carleman_second} as 
\begin{multline*}
(\delta \tau)^{1/2} \norm{\partial_{\nu} v}_{L^2(S_-)} + \norm{\partial_{\nu} v}_{L^2(S_0)} + \tau \norm{v} + \norm{\nabla v} \\
 \leq C_0 \norm{Lv} + C_0 \tau^{1/2} \norm{\partial_{\nu} v}_{L^2(S_+)}.
\end{multline*}
This is valid for any $\delta > 0$, provided that $\tau \geq \tau_0$ and $v \in C^{\infty}(\overline{\Omega})$ with $v|_{\partial \Omega} = 0$.

Consider the following subspace of $L^2(\Omega) \times L^2(S_+)$:
$$
X = \{ (Lv, \partial_{\nu} v|_{S_+}) \,;\, v \in C^{\infty}(\overline{\Omega}), \ v|_{\partial \Omega} = 0 \}.
$$
Any element of $X$ is uniquely represented as $(Lv, \partial_{\nu} v|_{S_+})$ where $v|_{\partial \Omega} = 0$ by the Carleman estimate. Define a linear functional $l: X \to \mC$ by 
$$
l(Lv, \partial_{\nu} v|_{S_+}) = ( v, f )_{L^2(\Omega)} - ( \partial_{\nu} v, f_- )_{L^2(S_- \cup S_0)}.
$$
By the Carleman estimate, we have 
\begin{align*}
\abs{l(Lv, \partial_{\nu} v|_{S_+})} &\leq \norm{v} \norm{f} + \norm{\partial_{\nu} v}_{L^2(S_-)} \norm{f_-}_{L^2(S_-)} \\
 &\qquad + \norm{\partial_{\nu} v}_{L^2(S_0)} \norm{f_-}_{L^2(S_0)} \\
 &\leq C_0 (\tau^{-1} \norm{f} + (\delta \tau)^{-1/2} \norm{f_-}_{L^2(S_-)} + \norm{f_-}_{L^2(S_0)}) \\
 &\qquad \times (\norm{Lv} + \tau^{1/2} \norm{\partial_{\nu} v}_{L^2(S_+)}).
\end{align*}
The Hahn-Banach theorem implies that $l$ extends to a continuous linear functional $\bar{l}: L^2(\Omega) \times \tau^{-1/2} L^2(S_+) \to \mC$ such that 
$$
\norm{\bar{l}} \leq C_0 (\tau^{-1} \norm{f} + (\delta \tau)^{-1/2} \norm{f_-}_{L^2(S_-)} + \norm{f_-}_{L^2(S_0)}).
$$
By the Riesz representation theorem, there exist functions $u \in L^2(\Omega)$ and $u_+ \in L^2(S_+)$ satisfying $\bar{l}(w, w_+) = (w, u)_{L^2(\Omega)} + (w_+, u_+)_{L^2(S_+)}$. Moreover, 
\begin{multline*}
\norm{u}_{L^2(\Omega)} + \tau^{-1/2} \norm{u_+}_{L^2(S_+)} \\
\leq C_0 (\tau^{-1} \norm{f} + (\delta \tau)^{-1/2} \norm{f_-}_{L^2(S_-)} + \norm{f_-}_{L^2(S_0)}).
\end{multline*}
If $v \in C^{\infty}(\overline{\Omega})$ and $v|_{\partial \Omega} = 0$, we have 
\begin{align}
(Lv, u)_{L^2(\Omega)} + (\partial_{\nu} v, u_+)_{L^2(S_+)} &= ( v, f )_{L^2(\Omega)} \label{lvu_solution} \\
 &\qquad - ( \partial_{\nu} v, f_- )_{L^2(S_- \cup S_0)}. \notag
\end{align}
Choosing $v$ compactly supported in $\Omega$, it follows that $L^* u = f$, or 
$$
e^{-\tau \phi}(-\Delta + q) (e^{\tau \phi} u) = f \quad \text{in } \Omega.
$$
Furthermore, $e^{\tau \phi} u \in H_{\Delta}(\Omega)$.

If $w, v \in C^{\infty}(\overline{\Omega})$ with $v|_{\partial \Omega} = 0$, an integration by parts gives 
$$
(Lv, w) = -(e^{-\tau \phi} \partial_{\nu} v, e^{\tau \phi} w)_{L^2(\partial \Omega)} + (v, L^* w).
$$
Given our solution $u$, we choose $u_j \in C^{\infty}(\overline{\Omega})$ so that $e^{\tau \phi} u_j \to e^{\tau \phi} u$ in $H_{\Delta}(\Omega)$. Applying the above formula with $w = u_j$ and taking the limit, we see that 
$$
(Lv, u) = -(e^{-\tau \phi} \partial_{\nu} v, e^{\tau \phi} u)_{L^2(\partial \Omega)} + (v, L^* u)
$$
for $v \in C^{\infty}(\overline{\Omega})$ with $v|_{\partial \Omega} = 0$. Combining this with \eqref{lvu_solution}, using that $L^* u = f$, gives 
$$
( \partial_{\nu} v, f_- )_{L^2(S_- \cup S_0)} + (\partial_{\nu} v, u_+)_{L^2(S_+)} = (e^{-\tau \phi} \partial_{\nu} v, e^{\tau \phi} u)_{L^2(\partial \Omega)}.
$$Since $\partial_{\nu} v$ can be chosen arbitrarily, it follows that 
$$
e^{\tau \phi} u|_{S_- \cup S_0} = e^{\tau \phi} f_-, \quad e^{\tau \phi} u|_{S_+} = e^{\tau \phi} u_+.
$$
We also see that $e^{\tau \phi} u|_{\partial \Omega} \in L^2(\partial \Omega)$.
\end{proof}

\section{Complex geometrical optics} \label{sec_cgo}

We will now describe a construction of CGO solutions that is relevant for the proof of Theorem \ref{thm_ks1}. Suppose that $\Omega \subset \mR \times \Omega_0$ where $\Omega_0$ is a bounded open set with $C^{\infty}$ boundary in $\mR^2$. Let $\phi(x) = x_1$, and let $E$ be an open subset of $\partial \Omega_0$ such that $\Gamma_i$ satisfies 
$$
\Gamma_i \subset \mR \times (\partial \Omega_0 \setminus E).
$$
Let also $q \in L^{\infty}(\Omega)$. We wish to construction a solution $u \in H_{\Delta}(\Omega)$ of the equation 
$$
(-\Delta+q)u= 0 \text{ in } \Omega
$$
where 
$$
u = e^{-s x_1}(m + r_0)
$$
and $s$ is a \emph{slightly complex large frequency}, 
$$
s = \tau + i\lambda
$$
where $\tau > 0$ (eventually $\tau \to \infty$) and $\lambda \in \mR$ is fixed. The use of a slightly complex frequency instead of a real frequency allows to introduce another real parameter $\lambda$ in the CGO solutions, which makes it possible to employ the Fourier transform in the $x_1$ variable.

Inserting the ansatz for $u$ in the equation, we need to solve 
$$
e^{sx_1}(-\Delta+q)(e^{-s x_1} r_0) = f \text{ in } \Omega
$$
where 
$$
f = -e^{sx_1}(-\Delta+q)(e^{-s x_1} m) = -(-\Delta + 2s \partial_1 - s^2 + q)m.
$$
It will be useful to look for an amplitude $m$ independent of $x_1$, so $m = m(x')$ where $x = (x_1,x')$ and $x' = (x_2,x_3) \in \mR^2$. Then $f$ has the simpler form 
$$
f = (-\Delta_{x'} - s^2 + q)m.
$$

We seek for an amplitude $m \in C^2(\overline{\Omega}_0)$ satisfying 
$$
\norm{(-\Delta_{x'}-s^2) m}_{L^2(\Omega_0)} = O(1), \quad \norm{m}_{L^2(\Omega_0)} = O(1),
$$
and the boundary values should satisfy 
$$
\norm{m}_{L^2(E)} = O(1), \quad \norm{m}_{L^2(\partial \Omega_0 \setminus E)} = o(1)
$$
as $\tau \to \infty$. These conditions have been chosen to be compatible with Proposition \ref{prop_carleman_solvability}, and they can be interpreted so that $m$ should be an \emph{approximate eigenfunction}, or \emph{quasimode}, of the Laplacian in $\Omega_0$ with frequency $s$. In fact, we will arrange so that 
$$
m|_{\partial \Omega_0 \setminus E} = 0.
$$
If we can find such an $m$â then Proposition \ref{prop_carleman_solvability} together with the fact that the inaccessible part $\Gamma_i$ satisfies $\Gamma_i \subset \mR \times (\partial \Omega_0 \setminus E)$ will allow to find a correction term $r$ with $\norm{r}_{L^2(\Omega)} = o(1)$ as $\tau \to \infty$.

Let us now describe one construction of such an $m$. We choose a straight line $\gamma$ in $\mR^2$ that intersects $\overline{\Omega}_0$ but does not meet $\partial \Omega_0 \setminus E$, and will construct an amplitude $m$ that concentrates on $\gamma$. We look for $m$ in the form 
$$
m(x') = e^{is \psi(x')} a(x').
$$
A computation gives that 
$$
(-\Delta_{x'}-s^2) m = e^{is \psi} \left[ s^2 (\abs{\nabla \psi}^2-1) a - is (2\nabla \psi \cdot \nabla a + (\Delta_{x'} \psi)a ) - \Delta_{x'} a\right].
$$
It is enough to choose $\psi$ and $a$ so that the following equations are valid in $\Omega_0$: 
$$
\abs{\nabla \psi}^2 = 1, \qquad 2\nabla \psi \cdot \nabla a + (\Delta_{x'} \psi)a = 0.
$$
The first equation is an eikonal equation, and distance functions are solutions. Choose a ball $\hat{\Omega}_0$ with $\Omega_0 \subset \subset \hat{\Omega}_0$, and choose some point $x_0' \in \hat{\Omega}_0 \setminus \overline{\Omega}_0$ that lies the line $\gamma$. Define 
$$
\psi(x') = \abs{x'-x_0'}, \quad x' \in \overline{\Omega}_0.
$$
Then $\psi \in C^{\infty}(\overline{\Omega}_0)$ satisfies $\abs{\nabla \psi} = 1$. The second equation above is a transport equation for $a$, and has the solution 
$$
a(r,\theta) = r^{-1/2} b(\theta)
$$
where $(r,\theta)$ are polar coordinates in $\mR^2$ with center at $x_0'$, and $b$ is any function in $C^{\infty}(S^1)$. (Note that $\psi(r,\theta) = r$ in these coordinates.) The point is that if the line $\gamma$ is given in the $(r,\theta)$ coordinates by $r \mapsto (r,\theta_0)$, then choosing $b$ supported very close to $\theta_0$ and independent of $\tau$ will result in the boundary conditions 
$$
\norm{m}_{L^2(E)} = O(1), \quad m|_{\partial \Omega_0 \setminus E} = 0
$$
as $\tau \to \infty$.

Combining the above amplitude construction with Proposition \ref{prop_carleman_solvability} results in the following existence result for CGO solutions. (See \cite[Section 6]{KSa} for the full details of the proof.)

\begin{prop} \label{prop_cgo_existence}
Suppose that $\Omega \subset \mR \times \Omega_0$ where $\Omega_0$ is a bounded open set with $C^{\infty}$ boundary in $\mR^2$. Let $\phi(x) = \pm x_1$ and decompose $\partial \Omega_{\rm tan}(\phi)$ as $\Gamma_a \cup \Gamma_i$ where the closed set $\Gamma_i$ satisfies, for some open subset $E$ of $\partial \Omega_0$, 
$$
\Gamma_i \subset \mR \times (\partial \Omega_0 \setminus E).
$$
Let also $q \in L^{\infty}(\Omega)$. Choose a ball $\hat{\Omega}_0$ with $\Omega_0 \subset \subset \hat{\Omega}_0$, let $x_0' \in \hat{\Omega}_0 \setminus \overline{\Omega}_0$, and let $(r,\theta)$ be polar coordinates in $\mR^2$ with center at $x_0'$. Let $\theta_0 \in S^1$ be such that the line $r \mapsto (r,\theta_0)$ meets $\overline{\Omega}_0$ but not $\partial \Omega_0 \setminus E$. There exists a solution $u \in H_{\Delta}(\Omega)$ of the equation 
$$
(-\Delta+q)u = 0 \quad \text{in } \Omega
$$
having the form 
$$
u = e^{-s\phi}(e^{isr} r^{-1/2} b(\theta) + r_0)
$$
where $b \in C^{\infty}(S^1)$ is supported very close to $\theta_0$, 
$$
\norm{r_0}_{L^2(\Omega)} = o(1)
$$
as $\tau \to \infty$, and 
$$
\supp(u|_{\partial \Omega}) \subset \Gamma
$$
if $\Gamma$ is any open set in $\partial \Omega$ with $\partial \Omega_{\mp}(\phi) \cup \Gamma_a \subset \Gamma$.
\end{prop}

\section{Uniqueness results} \label{sec_uniqueness}

We now describe how to complete the outline given in Section \ref{sec_strategy} and prove Theorem \ref{thm_ks1} using an injectivity result for a certain transform (in the end of the section we discuss briefly the proofs of Theorems \ref{thm_ks2}--\ref{thm_ks4}). If $f$ is a piecewise continuous compactly supported function in $\mR^3$, the relevant transform of $f$ is given by 
$$
Tf(\lambda,\gamma) =  \int_{-\infty}^{\infty} e^{-2\lambda t} \left[ \int_{-\infty}^{\infty} e^{-2i\lambda x_1} f(x_1,\gamma(t))  \,dx_1 \right] \,dt,
$$
where $\lambda \in \mR$ and $\gamma$ is a line in $\mR^2$. That is, we are taking the Fourier transform of $f$ in the $x_1$ variable with frequency $2\lambda$ and then taking the attenuated X-ray transform, with constant attenuation $-2\lambda$, along lines in $\mR^2_{x'}$. We choose the parametrization $\gamma(t) = \sigma \omega_{\perp} + t \omega$, where $\omega \in S^1$ is the direction vector of the line $\gamma$, $\omega_{\perp}$ is the counterclockwise rotation of $\omega$ by $90$ degrees, and $\sigma$ is the signed distance to the origin. (The choice of parametrization will not be too important below.)

\begin{prop} \label{prop_ks1_transform}
Suppose that $\Omega \subset \mR \times \Omega_0$ where $\Omega_0$ is a bounded open set with $C^{\infty}$ boundary in $\mR^2$. Let $\phi(x) = x_1$, and let $E$ be an open subset of $\partial \Omega_0$ such that $\Gamma_i$ satisfies 
$$
\Gamma_i \subset \mR \times (\partial \Omega_0 \setminus E).
$$
If $q_1, q_2 \in C(\overline{\Omega})$, if $\partial \Omega_- \cup \Gamma_a \subset \Gamma_D$ and $\partial \Omega_+ \cup \Gamma_a \subset \Gamma_N$, and if 
$$
C_{q_1}^{\Gamma_D,\Gamma_N} = C_{q_2}^{\Gamma_D,\Gamma_N},
$$
then 
$$
T(q_1-q_2)(\lambda,\gamma) = 0
$$
for all $\lambda \in \mR$ and for any line $\gamma$ in $\mR^2$ that does not meet $\partial \Omega_0 \setminus E$. (Here, we extend $q_1-q_2$ by zero to $\mR^2$.)
\end{prop}
\begin{proof}
By Lemma \ref{prop_integralidentity}, our assumptions imply that 
$$
\int_\Omega (q_1-q_2) u_1 \bar{u}_2 \,dx = 0
$$
for any $u_j \in H_{\Delta}(\Omega)$ satisfying $(-\Delta + q_1) u_1 = (-\Delta + \bar{q}_2) u_2 = 0$ in $\Omega$ and 
$$
\supp(u_1|_{\partial \Omega}) \subset \Gamma_D, \quad \supp(u_2|_{\partial \Omega}) \subset \Gamma_N.
$$
Fix $\lambda \in \mR$, let $s = \tau + i\lambda$ where $\tau > 0$ is sufficiently large, and fix a line $\gamma$ in $\mR^2$ that does not meet $\partial \Omega_0 \setminus E$. We use Proposition \ref{prop_cgo_existence} to find solutions $u_1$ and $u_2$ satisfying the above conditions and having the form 
\begin{align*}
u_1 &= e^{-s x_1} (e^{isr} r^{-1/2} b(\theta) + r_1), \\
u_2 &= e^{s x_1} (e^{isr} r^{-1/2} b(\theta) + r_2) \\
\end{align*}
where $(r,\theta)$ are polar normal coordinates in $\mR^2$ whose center is outside of $\Omega_0$ and such that the line $\gamma$ is given by $r \mapsto (r,\theta_0)$, $b \in C^{\infty}(S^1)$ is independent of $\tau$ and supported near $\theta_0$, and 
$$
\norm{r_j} = o(1) \text{ as } \tau \to \infty.
$$
We then have 
\begin{multline*}
0 = \lim_{\tau \to \infty} \int_\Omega (q_1-q_2) u_1 \bar{u}_2 \,dx \\
 = \int_{-\infty}^{\infty} \int_{0}^{\infty} \int_{S^1} (q_1-q_2)(x_1,r,\theta) e^{-2i\lambda x_1} e^{-2\lambda r} \abs{b(\theta)}^2 \,d\theta \,dr \,dx_1
\end{multline*}
Here we have extended $q_1-q_2$ by zero outside $\Omega$. The result follows by choosing $b$ to approximate a delta function at $\theta_0$.
\end{proof}

The uniqueness result now follows from the \emph{Helgason support theorem} for the X-ray transform \cite{Helgason}, which states that if $f$ is a piecewise continuous compactly supported function in $\mR^2$ that integrates to zero over any line that does not meet a compact convex set $K \subset \mR^2$, then $f|_{\mR^2 \setminus K} = 0$. The idea of reducing the attenuated X-ray transform to the unattenuated one comes from \cite{DKLSa}.

\begin{proof}[Proof of Theorem \ref{thm_ks1}]
Denote by $f$ the extension of $q_1-q_2$ by zero to $\mR^2$, and let 
$$
\hat{f}(\xi_1,x') = \int_{-\infty}^{\infty} e^{-ix_1 \xi_1} f(x_1,x') \,dx_1.
$$
Let also $K$ be the convex hull of $\partial \Omega_0 \setminus E$ in $\mR^2$. By Proposition \ref{prop_ks1_transform}, it is enough to show that if 
$$
Tf(\lambda,\gamma) = 0
$$
for all $\lambda \in \mR$ and for any line $\gamma$ that does not meet $K$, then $f(x_1,x') = 0$ whenever $x_1 \in \mR$ and $x' \in \mR^2 \setminus K$.

The condition $Tf(\lambda,\gamma) = 0$ implies that 
$$
\int_{-\infty}^{\infty} e^{-2\lambda t} \hat{f}(2\lambda,\gamma(t)) \,dt = 0
$$
for any line $\gamma$ that does not meet $K$. Setting $\lambda = 0$, the Helgason support theorem implies that $\hat{f}(0,x') = 0$ for $x' \in \mR^2 \setminus K$. We now differentiate the above identity with respect to $\lambda$ and set $\lambda = 0$ to obtain 
$$
\int_{-\infty}^{\infty} \left[ (-2t) \hat{f}(0,\gamma(t)) + 2 \frac{\partial \hat{f}}{\partial \xi_1}(0,\gamma(t)) \right] \,dt = 0.
$$
But we already saw that $\hat{f}(0,x') = 0$ for $x' \in \mR^2 \setminus K$. The Helgason support theorem then gives that 
$$
\frac{\partial \hat{f}}{\partial \xi_1}(0,x') = 0 \text{ for } x' \in \mR^2 \setminus K.
$$
Repeating this argument shows that 
$$
\left( \frac{\partial}{\partial \xi_1} \right)^k  \hat{f}(0,x') = 0 \text{ for } x' \in \mR^2 \setminus K
$$
for all $k \geq 0$. But $x_1 \mapsto f(x_1,x')$ is compactly supported, hence $\xi_1 \mapsto \hat{f}(\xi_1,x')$ is analytic, and we obtain that $\hat{f}(\xi_1,x') = 0$ for all $\xi_1 \in \mR$ and $x' \in \mR^2 \setminus K$. The result follows upon inverting the Fourier transform in $x_1$.
\end{proof}

Theorems \ref{thm_ks2} and \ref{thm_ks3}  can be proved by the same general argument as above. However, the fact that one has nonlinear limiting Carleman weights (the log weight and arg weight) leads, after a suitable conformal scaling, to a situation where the original domain $\overline{\Omega}$ with Euclidean metric is replaced by a compact Riemannian manifold $(M,g)$ with smooth boundary, and the condition 
$$
\Omega \subset \mR \times \Omega_0
$$
is replaced by the condition 
$$
(M,g) \subset (\mR \times M_0, g), \quad g = c(e \oplus g_0)
$$
where $c$ is a positive function (conformal factor), and the transversal domain $\overline{\Omega}_0$ with Euclidean metric is replaced by a transversal manifold $(M_0,g_0)$ that is compact with smooth boundary. Also the integrals over lines in $\overline{\Omega}_0$ are replaced by integrals over geodesics in $(M_0,g_0)$. This setup is similar to the results for the anisotropic Calder\'on problem in \cite{DKSaU}, \cite{DKLSa}. For the log weight, $(M_0,g_0)$ will be a domain in the hemisphere, and for the arg weight $(M_0,g_0)$ will be a domain in hyperbolic space. Correspondingly, one uses geodesics in the sphere and in hyperbolic space. In both cases, if $(M_0,g_0)$ is chosen to have strictly convex boundary, then it is a simple manifold with real-analytic metric and analogues of the Helgason support theorem are available. The full proofs are given in \cite{KSa}.

The proof of Theorem \ref{thm_ks4} requires a more general complex geometrical optics construction than the one described above. The reason is that one wants to get information on integrals along broken rays that may touch the inaccessible part $\Gamma_i$, and one then needs amplitudes in the complex geometrical optics solutions that concentrate along these broken rays and are very small on the inaccessible part. Such amplitudes may obtained from a Gaussian beam quasimode construction. This construction in connection with complex geometrical optics solutions is employed in \cite{DKLSa} for non-reflected geodesics (full data case) and in \cite{KSa} for reflected geodesics (partial data case), and we refer to these papers for the details.

\section{The linearized case} \label{sec_linearized}

Here we sketch the proof of Theorem \ref{thm_linearized}. A first reduction is to show that it suffices to prove a local version of the theorem (see \cite[Section 2]{DKSjU09}). The global version follows from the local one by using ideas in the spirit of the Runge approximation theorem, developed in an unpublished work of Alessandrini, Isozaki, and Uhlmann. Thus, matters reduce to proving the following ''local'' version.

\begin{thm} \label{thm_linearized_local}
Let $\Omega \subset \mR^n$, $n \geq 2$, be a bounded connected open set with $C^{\infty}$ boundary. Let $x_0 \in \partial \Omega$ and let $F$ be the complement of some open boundary neighborhood of $x_0$. There exists $\delta > 0$ such that given any $f \in L^{\infty}(\Omega)$, if we have the cancellation property 
$$
\int_{\Omega} f u_1 u_2 \,dx = 0
$$
for all $u_j \in C^{\infty}(\overline{\Omega})$ with $\Delta u_j = 0$ in $\Omega$ and $u_j|_F = 0$, then $f = 0$ in $B(x_0,\delta) \cap \Omega$.
\end{thm}

The next step is to use conformal transformations (in particular Kelvin transforms) of harmonic functions to reduce to the following situation: $x_0 = 0$, $\Omega \subset \{ x \in \mR^n \,;\, \abs{x+e_1} < 1 \}$ where $e_1 = (1,0,\ldots,0)$ is the first coordinate vector, the tangent hyperplane to $0$ is given by $\{ x_1 = 0 \}$, and $F \subset \{ x \in \partial \Omega \,;\, x_1 \leq -2c \}$ for some $c > 0$. See \cite[Section 3]{DKSjU09} for this reduction.

From this point on, the proof is inspired by the proof of Kashiwara's ''watermelon theorem'' in analytic microlocal analysis. We introduce the Segal-Bargmann transform of a function $f \in L^{\infty}(\mR^n)$ with compact support by 
$$
Tf(z) = \int_{\mR^n} e^{-\frac{1}{2h}(z-y)^2} f(y) \,dy, \quad z = x+i\xi \in \mC^n.
$$
The Segal-Bargmann transform is related to the microlocal analysis of analytic singularities of a distribution. We mention the a priori exponential bound 
\begin{equation} \label{linearized_estimate1}
\abs{Tf(z)} \leq (2\pi h)^{n/2} e^{\frac{1}{2h} \abs{\im z}^2} \norm{f}_{L^{\infty}}
\end{equation}
and the fact that if $f$ is supported in $\{ x_1 \leq 0 \}$, we can improve this to 
\begin{equation} \label{linearized_estimate2}
\abs{Tf(z)} \leq (2\pi h)^{n/2} e^{\frac{1}{2h} \abs{\im z}^2-\abs{\re z_1}^2} \norm{f}_{L^{\infty}}
\end{equation}
when $\re z_1 \geq 0$. Both of these bounds are straightforward. Note also that when we multiply by $(2\pi)^{-n/2}$, when $z = x \in \mR^n$, we obtain the Gaussian heat kernel, and hence if $f$ has compact support we have 
\begin{equation} \label{linearized_heatkernel_limit}
\lim_{h \to 0} \,(2\pi h)^{-n/2} Tf(x) = f(x) \quad \text{for a.e.~$x \in \mR^n$}.
\end{equation}

The strategy is to use the cancellation property (we extend $f$ to be $0$ outside of $\Omega$) to obtain an exponential decay of $Tf$ when $x \in \mR^n$ is close to $0$, thus yielding the vanishing of $f$ near $0$, as desired. In order to accomplish this, let us temporarily consider the case $n=2$ and define $\gamma = i e_1 + e_2 = (i, 1) \in \mC^2$. Note that $\{ \gamma, \bar{\gamma} \}$ is a basis of $\mC^2$, and that $\{ \zeta \in \mC^2 \,;\, \zeta \cdot \zeta = 0 \}$ is the union of two complex lines $\mC \gamma \cup \mC \overline{\gamma}$. Hence, it is easy to see if $\eps > 0$ is small enough, then any $z \in \mC^2$ with $\abs{z-2i e_1} < 2\eps$ may be decomposed as a sum of the form 
\begin{equation} \label{linearized_z_decomposition}
z = \zeta + \eta, \quad \zeta \cdot \zeta = \eta \cdot \eta = 0,
\end{equation}
where $\abs{\zeta-\gamma} \leq C \eps$, $\abs{\eta+\bar{\gamma}} \leq C \eps$. This last fact extends to $\mC^n$, $n \geq 2$, see \cite[Section 3]{DKSjU09}.

Recall that the exponentials $e^{-i\frac{\zeta \cdot x}{h}}$, $\zeta \in \mC^n$â $\zeta \cdot \zeta = 0$, are harmonic functions. We need to modify these by adding correction terms to obtain harmonic functions $u$ satisfying the boundary requirement $u|_F = 0$. Let $\chi \in C^{\infty}_c(\mR^n)$ be a cutoff function which is supported in $\{ x_1 \leq -c \}$ and equals $1$ on $\{ x_1 \leq -2c \}$. Consider the solution $w = w(\,\cdot\,;\zeta)$ of the Dirichlet problem 
$$
\Delta w = 0 \text{ in } \Omega, \quad w|_{\partial \Omega} = -e^{-i\frac{\zeta \cdot x}{h}} \chi|_{\partial \Omega}.
$$
We have the following bound on $w$:
\begin{align} \label{linearized_w_estimate}
\norm{w}_{H^1(\Omega)} &\leq C_1 \norm{e^{-i\frac{\zeta \cdot x}{h}} \chi}_{H^{1/2}(\partial \Omega)} \\
 &\leq C_2(1+h^{-1} \abs{\zeta})^{1/2} e^{-c\frac{\im \zeta_1}{h}} e^{\frac{1}{h} \abs{\im \zeta'}} \notag
\end{align}
where $\im \zeta_1 \geq 0$, from the choice of $\chi$ and our normalization of $\Omega$.

We have the cancellation property 
$$
\int_{\Omega} f(x) u(x,\zeta) u(x,\eta) \,dx = 0, \quad \zeta \cdot \zeta = \eta \cdot \eta = 0,
$$
where $u(x,\zeta) = e^{-i\frac{\zeta \cdot x}{h}} + w(x,\zeta)$, which is a harmonic function in $C^{\infty}(\overline{\Omega})$ and satisfies $u|_F = 0$. The identity above can be rewritten as 
\begin{multline*}
\int_{\Omega} f(x) e^{-\frac{i}{h} x \cdot (\zeta+\eta)} \,dx = -\int_{\Omega} f(x) e^{-\frac{i}{h} x \cdot \zeta} w(x,\eta) \,dx \\
- \int_{\Omega} f(x) e^{-\frac{i}{h} x \cdot \eta} w(x,\zeta) \,dx - \int_{\Omega} f(x) w(x,\zeta) w(x,\eta) \,dx.
\end{multline*}
Using the estimate \eqref{linearized_w_estimate}, we obtain 
\begin{multline*}
\left| \int_{\Omega} f(x) e^{-\frac{i}{h} x \cdot (\zeta+\eta)} \,dx \right| \leq C_3 \norm{f}_{L^{\infty}} (1+h^{-1}\abs{\eta})^{1/2} (1+h^{-1}\abs{\zeta})^{-1/2} \\
 \times e^{-\frac{c}{h} \min \{ \im \zeta_1, \im \eta_1 \}} e^{\frac{1}{h}(\abs{\im \zeta'}+\abs{\im \eta'})},
\end{multline*}
when $\im \zeta_1 \geq 0$, $\im \eta_1 \geq 0$, and $\zeta \cdot \zeta = \eta \cdot \eta = 0$. In particular, if $\abs{\zeta - a \gamma} \leq C \eps a$, $\abs{\eta + a \bar{\gamma}} < C \eps a$, $\eps \leq 1/(2C)$, then 
$$
\left| \int_{\Omega} f(x) e^{-\frac{i}{h} x \cdot (\zeta+\eta)} \,dx \right| \leq C_4 h^{-1} \norm{f}_{L^{\infty}(\Omega)} e^{-\frac{ca}{2h}} e^{\frac{2C\eps a}{h}}.
$$
Hence if $z \in \mC^n$ and $\abs{z-2i a e_1} < 2\eps a$, $\eps$ small enough, using a rescaled version of the decomposition \ref{linearized_z_decomposition} gives 
\begin{equation} \label{linearized_rescaled_estimate}
\left| \int_{\Omega} f(x) e^{-\frac{i}{h} x \cdot z} \,dx \right| \leq C_4 h^{-1} \norm{f}_{L^{\infty}(\Omega)} e^{-\frac{ca}{h}} e^{\frac{2C\eps a}{h}}.
\end{equation}

To relate the last estimate to the given estimates on the Segal-Bargmann transform of $f$, we use the well-known formula 
$$
e^{-\frac{1}{2h}(z-y)^2} = e^{-\frac{z^2}{2h}} (2\pi h)^{-n/2} \int e^{-\frac{t^2}{2h}} e^{-\frac{i}{h} y \cdot (t+iz)} \,dt
$$
which gives 
$$
Tf(z) = (2\pi h)^{-n/2} \int \int e^{-\frac{1}{2h}(z^2+t^2)} e^{-\frac{i}{h} y \cdot (t+iz)} f(y) \,dt \,dy.
$$
For our $f$, supported in $\Omega$ and verifying the cancellation of the integral in our hypothesis, the estimate \eqref{linearized_rescaled_estimate} and the formula above allow us to improve the estimate \eqref{linearized_estimate2}. If $\re z_1 \geq 0$, then 
\begin{multline*}
\abs{Tf(z)} \leq (2\pi h)^{-n/2} \int e^{\frac{1}{2h}(\abs{\im z}^2-\abs{\re z}^2-t^2)} \left| \int e^{-\frac{i}{h} y \cdot (t+iz)} \,dy \right| \,dt \\
 \leq \frac{e^{\frac{1}{2h}(\abs{\im z}^2-\abs{\re z}^2)}}{(2\pi h)^{n/2}} \bigg[ \int_{\abs{t} \leq \eps a} e^{-\frac{t^2}{2h}} \left| \int e^{-\frac{i}{h} y \cdot (t+iz)} f(y) \,dy \right| \,dt \\
  + \int_{\abs{t} \geq \eps a} e^{-\frac{t^2}{2h}} \left| \int e^{-\frac{i}{h} y \cdot (t+iz)} f(y) \,dy \right| \,dt \bigg]
\end{multline*}
Using \eqref{linearized_rescaled_estimate} with $z$ replaced by $t+iz$, when $\abs{z-2a e_1} < \eps a$ and $\abs{t} \leq \eps a$ we obtain 
\begin{multline}
\abs{Tf(z)} \leq e^{\frac{1}{2h}(\abs{\im z}^2-\abs{\re z}^2)} \bigg[ \sup_{\abs{t} \leq \eps a} \left| \int e^{-\frac{i}{h} y \cdot (t+x)} f(y) \,dy \right| \\
 + \sqrt{2} e^{\frac{1}{h}\abs{\re z'}} e^{-\frac{\eps^2 a^2}{4h}} \int \abs{f(y)} \,dy \bigg] \\
 \leq C_5 h^{-1} \norm{f}_{L^{\infty}(\Omega)} e^{\frac{1}{2h}(\abs{\im z}^2-\abs{\re z}^2)} \left[ e^{-\frac{ac}{2h}} e^{\frac{2C \eps a}{h}} + e^{-\frac{\eps^2 a^2}{4h}} e^{\frac{\eps a}{h}} \right],
\end{multline}
provided that $\abs{z-2a e_1} < \eps a$. Choosing $\eps < c/(8C)$ and $a > (c+4\eps)/(\eps^2)$, we obtain the bound 
$$
\abs{Tf(z)} \leq 2 C_5 h^{-1} \norm{f}_{L^{\infty}(\Omega)} e^{\frac{1}{2h}(\abs{\im z}^2-\abs{\re z}^2 - \frac{ca}{2}}).
$$
Combining \eqref{linearized_estimate1}, \eqref{linearized_estimate2} and \eqref{linearized_rescaled_estimate} we have 
$$
e^{-\frac{\Phi(z_1)}{2h}} \abs{Tf(z_1,x')} \leq \frac{C}{h} \norm{f}_{L^{\infty}(\Omega)} \left\{ \begin{array}{ll} 1, & z_1 \in \mC, \\ e^{-\frac{ca}{4h}}, & \abs{z_1-2a} \leq \frac{\eps a}{2}, \abs{x'} < \frac{\eps a}{2} \end{array} \right.
$$
where $x' \in \mR^{n-1}$ and 
$$
\Phi(z_1) = \left\{ \begin{array}{ll} (\im z_1)^2, & \text{when } \re z_1 \leq 0, \\ (\im z_1)^2-(\re z_1)^2, & \text{when } \re z_1 \geq 0. \end{array} \right.
$$
Now, inspired by the proof of the ''watermelon theorem'', we use the following:

\begin{lemma} \label{lemma_linearized_watermelon}
Let $b, L > 0$. Let $F$ be an entire function in $\mC$ such that 
$$
e^{-\frac{\Phi(s)}{h}} \abs{F(s)} \leq \left\{ \begin{array}{ll} 1, & s \in \mC, \\ e^{-\frac{c}{2h}}, & \text{when } \abs{s-L} \leq b. \end{array} \right.
$$
Then for all $r \geq 0$, there exist $c', \delta > 0$ (independent of $F$) such that 
$$
\abs{F(s)} \leq e^{-\frac{c'}{2h}+\frac{(\im s)^2}{2h}}
$$
when $\abs{\re s} \leq \delta$ and $\abs{\im s} \leq r$.
\end{lemma}

The proof of the lemma rests on a harmonic majorization argument, which exploits the subharmonicity of $-(\im s)^2+(\re s)^2$ (see \cite[Lemma 4.1 and Remark 4.2]{DKSjU09} for the details of the proof). Next we apply Lemma \ref{lemma_linearized_watermelon} to 
$$
F(s) = \frac{h Tf(s,x')}{C \norm{f}_{L^{\infty}(\Omega)}}.
$$
We obtain in particular that 
$$
\abs{Tf(x)} \leq C h^{-1} \norm{f}_{L^{\infty}(\Omega)} e^{-\frac{c'}{2h}}, \quad x \in \Omega, \abs{x_1} < \delta
$$
for $\delta$ small. Multiplying by $(2\pi h)^{-n/2}$ and letting $h \to 0$, using \eqref{linearized_heatkernel_limit}, we deduce that $f(x) = 0$ for $x \in \Omega$, $-\delta \leq x_1 \leq 0$. The completes the proof of Theorem \ref{thm_linearized_local} and hence of Theorem \ref{thm_linearized}.

\section{Open problems} \label{sec_open}

The following is a list of some questions related to the partial data problem that are open, as far as we know, at the time of writing this. The first question concerns the local data problem for $n \geq 3$. (As discussed above, this result for $n=2$ is known at least for sufficiently regular coefficients.)

\bigskip

\noindent {\bf Question 1. (Local data in dimensions $n \geq 3$)} \ 
If $\Omega$ is a bounded domain in $\mR^n$, $n \geq 3$, if $\Gamma$ is an arbitrary nonempty open subset of $\partial \Omega$, and if $q_1, q_2 \in L^{\infty}(\Omega)$, show that $C_{q_1}^{\Gamma,\Gamma} = C_{q_2}^{\Gamma,\Gamma}$ implies $q_1 = q_2$.

\bigskip

The next question concerns data on disjoint sets. The general case is open even for the linearized problem in any dimension. Partial results for $n=2$ are given in \cite{IUY_disjoint}.

\bigskip

\noindent {\bf Question 2. (Data on disjoint sets in dimensions $n \geq 2$)} \ 
If $\Omega$ is a bounded domain in $\mR^n$, $n \geq 2$, if $\Gamma_D$ and $\Gamma_N$ are arbitrary disjoint open subsets of $\partial \Omega$, and if $q_1, q_2 \in L^{\infty}(\Omega)$, show that $C_{q_1}^{\Gamma_D,\Gamma_N} = C_{q_2}^{\Gamma_D,\Gamma_N}$ implies $q_1 = q_2$.

\bigskip

One can also ask for optimal regularity conditions for the coefficients in partial data results. If $n=2$ the full data result is known for $L^{\infty}$ conductivities \cite{AP}.

\bigskip

\noindent {\bf Question 3. (Local data for nonsmooth conductivities)} \ 
If $\Omega$ is a bounded domain in $\mR^2$, if $\Gamma$ is an arbitrary nonempty open subset of $\partial \Omega$, and if $\gamma_1, \gamma_2 \in L^{\infty}(\Omega)$ satisfy $\gamma_1, \gamma_2 \geq c > 0$, show that $C_{\gamma_1}^{\Gamma,\Gamma} = C_{\gamma_2}^{\Gamma,\Gamma}$ implies $\gamma_1 = \gamma_2$.

\bigskip

The next question concerns optimal stability for partial data results. In the full data case it is known that in general one has a logarithmic modulus of continuity for determining an unknown coefficient from boundary measurements, and this result is optimal. We refer to the survey \cite{Alessandrini_survey}. Stability results for partial boundary measurements based on the reflection approach \cite{HeckWang2}, \cite{Caro_stability} and in the case where the coefficient is known near the boundary \cite{Alessandrini_partial} also have logarithmic stability. However, available results for the Carleman estimate approach \cite{HeckWang}, \cite{Tzou}, \cite{CDR} seem to involve weaker moduli of continuity (log log or worse), and one can ask if logarithmic stability still holds.

\bigskip

\noindent {\bf Question 4. (Optimal stability for the Carleman estimate approach)} \ 
What is the optimal stability for the partial data uniqueness result of \cite{KSU}?

\bigskip

Finally, there are many open questions related to partial data for elliptic systems. In the introduction we mentioned results for systems when $n=2$ and for the reflection approach when $n \geq 3$, but it seems that there are no partial data results for systems using the Carleman estimate approach in $n \geq 3$ except for \cite{SaloTzou}. As an example, one can consider the time-harmonic Maxwell equations for which the full data result is known \cite{OPS}.

\bigskip

\noindent {\bf Question 5. (Carleman estimate approach for systems)} \ 
Prove an analogue of the partial data result of \cite{KSU} for the time-harmonic Maxwell equations with scalar coefficients as in \cite{OPS}.

\bibliographystyle{alpha}

\begin{thebibliography}{DKSaU09}
\bibitem[AGTU11]{GT_system} P. Albin, C. Guillarmou, L. Tzou, G. Uhlmann, \textit{{Inverse boundary problems for systems in two dimensions}}, Ann. Inst. H. Poincar\'e (to appear), arXiv:1105.4565.
\bibitem[Al07]{Alessandrini_survey} G. Alessandrini, \textit{{Open issues of stability for the inverse conductivity problem}}, J. Inverse Ill-Posed Probl. \textbf{15} (2007), 451--460.
\bibitem[AK12]{Alessandrini_partial} G. Alessandrini, K. Kim, \emph{{Single-logarithmic stability for the Calder{\'o}n problem with local data}}, J. Inverse Ill-posed Probl. \textbf{20} (2012), 389--400.
\bibitem[AU04]{AU} H. Ammari, G. Uhlmann, \emph{{Reconstruction of the potential from partial Cauchy data for the Schr{\"o}dinger equation}}, Indiana Math J. \textbf{53} (2004), 169--184.
\bibitem[ALP05]{ALP} K. Astala, M. Lassas, L. P\"aiv\"arinta, \textit{Calder{\'o}n's inverse problem for anisotropic conductivity in the plane}, Comm. PDE \textbf{30} (2005), 207--224.
\bibitem[AP06]{AP} K. Astala, L. P\"aiv\"arinta, \textit{Calder{\'o}n's inverse conductivity problem in the plane}, Ann. of Math. \textbf{163} (2006), 265--299.
\bibitem[Be09]{Benjoud} H. Ben Joud, \textit{{A stability estimate for an inverse problem for the Schr{\"o}dinger equation in a magnetic field from partial boundary measurements}}, Inverse Problems \textbf{25} (2009), 045012.
\bibitem[Bu08]{Bu} A.L. Bukhgeim, \textit{Recovering a potential from Cauchy data in the two-dimensional case}, J. Inverse Ill-posed Probl. \textbf{16} (2008), 19--34.
\bibitem[BU01]{BU} A.L. Bukhgeim, G. Uhlmann, \textit{Recovering a potential from partial Cauchy data}, Comm. PDE \textbf{27} (2002), 653--668.
\bibitem[Ca80]{C} A.P. Calder{\'o}n, \emph{On an inverse boundary value problem}, Seminar on Numerical Analysis and its Applications to Continuum Physics, Soc. Brasileira de Matem{\'a}tica, R{\'i}o de Janeiro, 1980.
\bibitem[Ca11]{Caro_stability} P. Caro, \textit{{On an inverse problem in electromagnetism with local data: stability and uniqueness}}, Inverse Probl. Imaging \textbf{5} (2011), 297--322.
\bibitem[CDR12]{CDR}
P. Caro, D. Dos Santos Ferreira, A. Ruiz, \textit{{Stability estimates for the Radon transform with restricted data and applications}}, preprint (2012), arXiv:1211.1887.
\bibitem[COS09]{COS}
P. Caro, P. Ola, M. Salo, \emph{Inverse boundary value problem for Maxwell equations
with local data}, Comm. PDE \textbf{34} (2009), 1425--1464.
\bibitem[Ch11]{Chung} F.J. Chung, \textit{{A partial data result for the magnetic Schrodinger inverse problem}}, preprint (2011), arXiv:1111.6658.
\bibitem[Ch12]{Chung_ND} F.J. Chung, \textit{{Partial data for the Neumann-Dirichlet map}}, preprint (2012), arXiv:1211.0211.
\bibitem[DKSaU09]{DKSaU}
D. Dos Santos Ferreira, C.E. Kenig, M. Salo, G. Uhlmann, \emph{{Limiting Carleman weights and anisotropic inverse problems}}, Invent. Math. \textbf{178} (2009), 119--171.
\bibitem[DKSjU07]{DKSjU07} D. Dos Santos~Ferreira, C.E. Kenig, J. Sj{\"o}strand, G. Uhlmann, \emph{{Determining a magnetic Schr{\"o}dinger operator from partial Cauchy data}}, Comm. Math. Phys. \textbf{271} (2007), 467--488.
\bibitem[DKSjU09]{DKSjU09} D. Dos Santos~Ferreira, C.E. Kenig, J. Sj{\"o}strand, G. Uhlmann, \emph{{On the linearized local Calder\'on problem}}, Math. Res. Lett. \textbf{16} (2009), 955--970.
\bibitem[DKLS13]{DKLSa} D. Dos Santos Ferreira, Y. Kurylev, M. Lassas, M. Salo, \emph{{The Calder{\'o}n problem in transversally anisotropic geometries}}, in preparation.
\bibitem[Es04]{Eskin_obstacles} G. Eskin, \emph{{Inverse boundary value problems for domains with several obstacles}}, Inverse Problems \textbf{20} (2004), 1497--1516.
\bibitem[Fa66]{Faddeev}
L.D. Faddeev, \emph{{Increasing solutions of the Schr\"odinger equation}}, Sov. Phys. Dokl. \textbf{10} (1966),
1033--1035.
\bibitem[Fa07]{Fathallah} I.K. Fathallah, \textit{{Stability for the inverse potential problem by the local Dirichlet-to-Neumann map for the Schr{\"o}dinger equation}}, Appl. Anal. \textbf{86} (2007), 899914.
\bibitem[Ge08]{Gebauer} B. Gebauer, \textit{{Localized potentials in electrical impedance tomography}}, Inverse Probl. Imaging \textbf{2} (2008), 251--269.
\bibitem[GT11a]{GT} C.~Guillarmou, L.~Tzou, \textit{{Calder{\'o}n inverse problem with partial data on Riemann surfaces}}, Duke Math. J. \textbf{158} (2011), 83--120.
\bibitem[GT11b]{GT_magnetic} C.~Guillarmou, L.~Tzou, \textit{{Identification of a connection from Cauchy data on a Riemann surface with boundary}}, Geom. Funct. Anal. \textbf{21} (2011), 393--418.
\bibitem[GT13]{GT_survey} C.~Guillarmou, L.~Tzou, \textit{{Survey on Calder{\'o}n inverse problem in 2 dimensions}}, Inside Out II (to appear).
\bibitem[HT11]{HT} B. Haberman, D. Tataru, \textit{{Uniqueness in Calderon's problem with Lipschitz conductivities}}, Duke Math. J. (to appear), arXiv:1108.6068.
\bibitem[H\"a98]{H98} P. H\"ahner, \textit{{A uniqueness theorem for an inverse scattering problem in an exterior domain}}, SIAM J. Math. Anal. \textbf{29} (1998), 1118--1128.
\bibitem[HW06]{HeckWang} H. Heck, J.-N. Wang, \textit{{Stability estimates for the inverse boundary value problem by partial Cauchy data}}, Inverse Problems \textbf{22} (2006), 1787.
\bibitem[HW07]{HeckWang2} H. Heck, J.-N. Wang, \textit{{Optimal stability estimate of the inverse boundary value problem by partial measurements}}, preprint (2007), arXiv:0708.3289.
\bibitem[He99]{Helgason} S.~Helgason, \textit{{The Radon transform}}, 2nd edition, Birkh\"auser, 1999.
\bibitem[HPS12]{HPS} N.~Hyv\"onen, P.~Piiroinen, O.~Seiskari, \emph{{Point measurements for a Neumann-to-Dirichlet map and the Calder{\'o}n problem in the plane}}, preprint (2012), arXiv:1204.0346.
\bibitem[Il12]{Ilmavirta}
J. Ilmavirta, \emph{{Broken ray tomography in the disk}}, preprint (2012), arXiv:1210.4354.
\bibitem[IUY10]{IUY} O.~Imanuvilov, G.~Uhlmann, M.~Yamamoto, \textit{{The Calder{\'o}n problem with partial data in two dimensions}}, J. Amer. Math. Soc. \textbf{23} (2010), 655--691.
\bibitem[IUY11a]{IUY_general} O.~Imanuvilov, G.~Uhlmann, M.~Yamamoto, \textit{{Determination of second-order elliptic operators in two dimensions from partial Cauchy data}}, Proc. Natl. Acad. Sci. USA \textbf{108} (2011), 467--472.
\bibitem[IUY11b]{IUY_disjoint} O.~Imanuvilov, G.~Uhlmann, M.~Yamamoto, \textit{{Inverse boundary value problem by measuring Dirichlet and Neumann data on disjoint sets}}, Inverse Problems \textbf{27} (2011), 085007.
\bibitem[IY12a]{IY_nonsmooth} O. Imanuvilov, M. Yamamoto, \textit{{Inverse boundary value problem for Schr{\"o}dinger equation in two dimensions}}, SIAM J. Math. Anal. \textbf{44} (2012), 1333--1339.
\bibitem[IY12b]{IY_systems} O. Imanuvilov, M. Yamamoto, \textit{{Inverse problem by Cauchy data for an arbitrary sub-boundary for systems of elliptic equations}}, Inverse Problems \textbf{28} (2012), 095015.
\bibitem[IY12c]{IY_3D} O. Imanuvilov, M. Yamamoto, \textit{{Inverse boundary value problem for Schr{\"o}dinger equation in cylindrical domain by partial boundary data}}, preprint (2012), arXiv:1211.1419.
\bibitem[Is88]{I88} V. Isakov, \textit{{On uniqueness of recovery of a discontinuous conductivity coefficient}}, Comm. Pure Appl. Math \textbf{41} (1988), 865--877.
\bibitem[Is07]{I}
V. Isakov, \emph{On uniqueness in the inverse conductivity problem with local data}, Inverse Probl. Imaging \textbf{1} (2007), 95--105.
\bibitem[KKL01]{KKL} A. Katchalov, Y. Kurylev, M. Lassas, \textit{Inverse Boundary Spectral Problems}, Monographs and Surveys in Pure and Applied Mathematics 123, Chapman Hall/CRC, 2001.
\bibitem[KS12]{KSa}
C.E. Kenig, M. Salo, \emph{{The Calder{\'o}n problem with partial data on manifolds and applications}}, preprint (2012), arXiv:1211.1054.
\bibitem[KSU11a]{KSaU}
C.E. Kenig, M. Salo, G. Uhlmann, \emph{{Inverse problems for the anisotropic Maxwell equations}}, Duke Math. J. \textbf{157} (2011), 369--419.
\bibitem[KSU11b]{KSaU_reconstruction}
C.E. Kenig, M. Salo, G. Uhlmann, \emph{{Reconstructions from boundary measurements on admissible manifolds}},
Inverse Probl. Imaging \textbf{5} (2011), 859--877.
\bibitem[KSU07]{KSU} C.E. Kenig, J. Sj\"ostrand, G. Uhlmann,  \textit{The Calder\'on problem with partial data}, Ann. of Math. \textbf{165} (2007), 567--591.
\bibitem[Kn06]{Knudsen}
K. Knudsen, \textit{{The Calder{\'o}n problem with partial data for less smooth conductivities}}, Comm. PDE \textbf{31} (2006), 57--71.
\bibitem[KS07]{KnudsenSalo}
K. Knudsen, M. Salo, \emph{Determining nonsmooth first order terms from partial boundary measurements}, Inverse Probl. Imaging \textbf{1} (2007), 349--369.
\bibitem[KV84]{KV} R. Kohn, M. Vogelius, \textit{{Determining conductivity by boundary measurements}}, Comm. Pure Appl. Math. \textbf{37} (1984), 289--298.
\bibitem[KV85]{KV2} R. Kohn, M. Vogelius, \textit{{Determining conductivity by boundary measurements. II. Interior results}}, Comm. Pure Appl. Math. \textbf{38} (1985), 644--667.
\bibitem[KLU12]{KLU} K.~Krupchyk, M.~Lassas, G.~Uhlmann, \emph{{Inverse problems with partial data for a magnetic Schr{\"o}dinger operator in an infinite slab and on a bounded domain}}, Comm. Math. Phys. \textbf{312} (2012), 87--126.
\bibitem[LO10]{LO1} M.~Lassas, L.~Oksanen, \emph{{Inverse problem for wave equation with sources and observations on disjoint sets}}, Inverse Problems \textbf{26} (2010), 085012.
\bibitem[LO12]{LO2} M.~Lassas, L.~Oksanen, \emph{{Inverse problem for the Riemannian wave equation with Dirichlet data and Neumann data on disjoint sets}}, preprint (2012), arXiv:1208.2105.
\bibitem[LiU10]{LiUhlmann} X.~Li, G.~Uhlmann, \emph{{Inverse problems with partial data in a slab}}, Inverse Probl. Imaging \textbf{4} (2010), 449--462.
\bibitem[Na96]{N_2D}
A.~Nachman, \emph{{Global uniqueness for a two-dimensional inverse boundary value problem}}, Ann. of Math. \textbf{143} (1996), 71--96.
\bibitem[NS10]{NS}
A.~Nachman, B.~Street, \emph{{Reconstruction in the Calder{\'o}n problem with partial data}}, Comm. PDE \textbf{35} (2010), 375--390.
\bibitem[OPS93]{OPS} P. Ola, L. P\"aiv\"arinta, E. Somersalo, \textit{An inverse boundary value problem in electrodynamics}, Duke Math. J. \textbf{70} (1993), 617--653.
\bibitem[ST10]{SaloTzou} M.~Salo, L.~Tzou, \emph{{Inverse problems with partial data for a Dirac system: a Carleman estimate approach}}, Adv. Math. \textbf{225} (2010), 487--513.
\bibitem[SU87]{SU} J.~Sylvester, G.~Uhlmann, \textit{A global uniqueness theorem for an inverse boundary value problem}, Ann. of Math. \textbf{125} (1987), 153--169.
\bibitem[Tz08]{Tzou} L. Tzou, \textit{{Stability estimate for the coefficients of magnetic Schr{\"o}dinger equation from full and partial boundary measurements}}, Comm. PDE \textbf{11} (2008), 1911--1952.
\bibitem[Uh98]{U_ICM} G. Uhlmann, \textit{{Inverse boundary value problems for partial differential equations}}, Doc. Math. J. DMV Extra Volume ICM III (1998), 77--86.
\bibitem[Uh09]{U_IP} G.~Uhlmann, \textit{Electrical impedance tomography and Calder{\'o}n's problem}, Inverse Problems \textbf{25} (2009), 123011.
\bibitem[Zh12]{Zhang} G.~Zhang, \emph{{Uniqueness in the Calder{\'o}n problem with partial data for less smooth conductivities}}, Inverse Problems \textbf{28} (2012), 105008.

\end{thebibliography}

\end{document}